\documentclass[reqno,10p]{amsproc}
\usepackage[utf8]{inputenc}
\usepackage{amssymb}
\usepackage{euscript}
\usepackage{extarrows}
\usepackage{mathrsfs}
\usepackage{a4wide}
\usepackage{setspace}
\usepackage{tikz}
\usetikzlibrary{positioning}
\usepackage{extarrows}
\usetikzlibrary{arrows.meta}

\newtheorem{thm}{Theorem}%[section]

\newtheorem{lem}[thm]{Lemma}
\newtheorem{pro}[thm]{Proposition}

\theoremstyle{remark}
\newtheorem{rem}[thm]{Remark}

\theoremstyle{definition}
\newtheorem{exa}[thm]{Example}
\newtheorem{dfn}[thm]{Definition}

\newcommand*{\card}[1]{\mathrm{card}(#1)}

\DeclareMathOperator{\D}{d\hspace{-0.25ex}}

\DeclareMathOperator{\paa}{{\mathsf{par}}}

\newcommand*{\ascr}{\mathscr A}

\newcommand*{\borel}[1]{{\mathfrak B}(#1)}
\newcommand*{\bscr}{\mathscr B}
\newcommand*{\cbb}{\mathbb C}

\newcommand*{\cfw}{C_{\phi,w}}
\newcommand*{\cfwn}[1]{C_{\phi^{#1},w_{#1}}}

\newcommand*{\cfrak}[1]{{\mathfrak{C}}_{#1}}

\newcommand*{\dz}[1]{{\EuScript D}(#1)}

\newcommand*{\esf}{\mathsf{E}}

\newcommand*{\efw}{\mathsf{E}_{\phi,w}}
\newcommand*{\efwn}[1]{{\mathsf E}_{#1}}

\newcommand*{\Ge}{\geqslant}

\newcommand*{\hh}{\mathcal H}

\newcommand*{\hsf}{{\mathsf h}}

\newcommand*{\hfw}{{\mathsf h}_{\phi,w}}

\newcommand*{\hsfn}[1]{{\mathsf h}_{#1}}

\newcommand*{\jd}[1]{\EuScript N(#1)}

\newcommand*{\Le}{\leqslant}

\newcommand*{\mno}[1]{{\mathsf M}_{#1}}
\newcommand*{\mcal}[1]{{\mathscr#1}}

\newcommand*{\nbb}{\mathbb N}

\newcommand*{\ogr}[1]{\boldsymbol B(#1)}
\newcommand*{\ob}[1]{{\EuScript R}(#1)}

\newcommand*{\opr}[1]{\boldsymbol L(#1)}

\newcommand*{\psf}{\mathsf{P}}

\newcommand*{\pfwn}[1]{{\mathsf{P}}_{{#1}}}

\newcommand*{\rbb}{\mathbb R}
\newcommand*{\rbop}{{\overline{\rbb}_+}}

\newcommand*{\smalloplus}{\raise0pt\hbox{$\scriptscriptstyle \oplus$}}

\newcommand*{\zbb}{\mathbb Z}

\begin{document}
\setstretch{1.2}
\title[Spectrally $n$-centered operators]{Spectrally $n$-centered operators. I}

\author[P.\ Budzy\'{n}ski]{Piotr Budzy\'{n}ski}
\address{Katedra Zastosowa\'{n} Matematyki, Uniwersytet Rolniczy w Krakowie, ul.\ Balicka 253c, 30-198 Kra\-k\'ow, Poland}
\email{piotr.budzynski@urk.edu.pl}

\date{\today}

\begin{abstract}
We introduce the concept of spectrally $n$-centered operators and study it within the framework of unbounded weighted composition operators in $L^2$-spaces. Based on this, we establish a characterisation of spectrally centered unbounded weighted composition operators, generalising Giselsson's criterion to the unbounded regime. Finally, we analyse various associated conditions, highlighting the delicate measure-theoretic anomalies.
\end{abstract}
\maketitle

\section{Introduction}

The study of structural properties of Hilbert space operators based on commutativity has long been driven by the spectral theorem for normal operators. An interesting point in this endeavor was the introduction of \emph{centered operators} by Morrel and Muhly in 1974 \cite{1974-sm-morrel-muhly}. A bounded operator $T$ on a Hilbert space $\hh$ is {\em centered} if the family of positive self-adjoint operators 
$$\{T^k T^{k*}, T^{m*} T^m \colon k, m \in \zbb_+\}$$ 
is mutually commutative. This condition effectively forces a centered operator to decompose into a direct sum of shifts and normal operators. The concept is related to \emph{weak centeredness} of operators\footnote{Weakly centered operators are also known as binormal operators.} (studied initially by Embry \cite{1966-pjm-embry, 1970-pjm-embry} and Campbell \cite{1972-pams-campbell, 1974-pjm-campbell}, and later by Ito, Yamazaki, Yanagida \cite{2004-ieot-ito-yamazaki-yanagida, 2004-jot-ito-yamazaki-yanagida}, and Paulsen, Pearcy, Petrović \cite{1992-jfa-petrovic, 1995-jfa-paulsen-pearcy-petrovic}), which only require commutativity of $T^*T$ and $TT^*$.

As shown independently by Campbell and Morrel (see \cite{1972-pams-campbell, 1971-cmbs-morrel}), every weakly centered hyponormal operator possesses a nontrivial invariant subspace. On the other hand, if an operator is centered, the existence of an invariant subspace follows from the Morrel--Muhly model without any additional assumptions. This naturally raises a question of whether there is a property between weak centeredness (combined with hyponormalty) and ``full'' centeredness that yields existence of an invariant subspace. In this paper, we bridge this gap by introducing and investigating {\em spectrally $n$-centered} operators. This notion is related to $n$-centeredness studied by Liu, Luo, Xu, and Zhang (see \cite{2019-bjma-liu-luo-xu, 2023-bjma-liu-xu-zhang})

While the concepts of centeredness and weak centeredness seem to be well-understood for bounded operators, extending them to the unbounded regime introduces difficulties. For unbounded operators, products such as $T^k T^{k*}$ and $T^{m*} T^m$ may have significantly smaller domain than the original operator $T$, rendering standard algebraic commutativity insufficient for obtaining interesting results. The remedy is to transition to \emph{spectral commutativity}. Instead of multiplying the operators directly, one requires the spectral measures of their moduli to commute. Such approach led to the recent investigation of \emph{spectrally weakly centered} and \emph{spectrally half-centered} unbounded weighted composition operators in $L^2$-spaces (see \cite{2025-arxiv-budzynski, 2026-arxiv-budzynski-01}). We use the same approach in the following paper. Unbounded weakly centered operators were also studied in \cite{2026-cherif-cherifa-mortad}, mostly by means of pointwise commutativity.

We carry our study within the framework of weighted composition operators in $L^2$-spaces. These operators, defined by 
$$\cfw f = w \cdot (f \circ \phi),$$ 
allow us to explicitly characterize abstract operator-theoretic properties in terms of Radon-Nikodym derivatives $\hsfn{k}$ and conditional expectations $\efwn{k}$. They form an excellent working ground for testing hypothesis, constructing examples, etc. We use them in hope to stimulate further research in the context of genereal Hilbert space operators. Theory of bounded and unbounded weighted composition operators in $L^2$-spaces has seen a significant development in recent years (see e.g., \cite{2014-ampa-budzynski-jablonski-jung-stochel, 2015-jfa-budzynski-jablonski-jung-stochel, 2018-lnim-budzynski-jablonski-jung-stochel, 2017-aim-budzynski-jablonski-jung-stochel, 2020-mn-benhida-budzynski-trepkowski})

Till the publication of \cite{2018-lnim-budzynski-jablonski-jung-stochel}, it was quite a common practice to treat weighted composition operators in $L^2$-spaces as if they all were products of multiplication operators and composition operators $\cfw = M_w C_\phi$, which, as detailed in \cite[Chapter 7]{2018-lnim-budzynski-jablonski-jung-stochel}, does not hold in general. Consequently, characterising various properties of weighted composition operators in the language of the Radon-Nikodym derivative $\hsf_\phi$ associated to $C_\phi$ leads to essential lack of generality. In particular, this can be said about the characterisation of centered weighted composition operators obtained in \cite{2018-mb-jabbarzadeh-bakhshkandi}. In this paper, we avoid any a priori restrictions related to the existence of $C_\phi$ and study weighted composition operators in $L^2$-spaces in full generality.

An underlying theme of this paper is the delicate contrast between the bounded and unbounded settings. As we demonstrate, algebraic or geometric conditions that perfectly characterize bounded centered weighted composition operators may fail for unbounded ones. 

Our main results are contained in Sections 4 and 5 of the paper. In Section 4, we characterize spectrally $n$-centered weighted composition operators. We show that for a fixed $n$, the equivalent condition is the commutativity of the range projections $\pfwn{k}$ with the modulus operators $\mno{l}$, which translates to the $\phi^{-k}(\ascr)$-measurability of the sequence of Radon-Nikodym derivatives $\hsf_l$ up to index $n$.  In Section 5, we dissect the relationship between spectral $n$-centeredness and other related conditions, revealing the measure-theoretic subtleties that distinguish the general theory from its bounded counterpart.

\section{Preliminaries}
We write $\zbb$, $\rbb$, and $\cbb$ for the sets of integers, real numbers, and complex numbers, respectively. By $\nbb$, $\zbb_+$, and $\rbb_+$ we denote the sets of positive integers, nonnegative integers, and nonnegative real numbers, respectively. $\rbop$ stands for $\rbb_+ \cup \{\infty\}$; $\borel{\cbb}$ and $\borel{\rbb}$ denote the $\sigma$-algebra of Borel subsets of $\cbb$ and $\rbb$, respectively. Given $n\in\nbb$, $J_n=\{1, 2, \ldots, n\}$, while $J_\infty=\nbb$.

Let $\hh$ be a (complex) Hilbert space. $\ogr{\hh}$ stands for a $\mathcal{C}^*$-algebra of bounded operators on $\hh$, while $\opr{\hh}$ stands for the linear space of (possibly unbounded) linear operators in $\hh$. Given $n\in\nbb\cup\{\infty\}$, we set
 \begin{align*}
{\mathfrak{C}}_{n}(\hh)=\big\{ A\in \opr{\hh}\colon \text{$A^k$ is closed and densely defined for } k\in J_n\big\}.
 \end{align*}
If no confusion arises, we drop the dependence on $\hh$ and write ${\mathfrak{C}}_{n}$ instead of ${\mathfrak{C}}_{n}(\hh)$.
 
Let $A\in \opr{\hh}$. Then $\dz{A}$, $\jd{A}$, $\ob{A}$, and $A^*$ stand for the domain, the kernel, the range, and the adjoint of $A$, respectively. A closed subspace $\mathcal{M}$ of $\hh$ is {\em invariant} for $A$ if  $A\big(\mathcal{M}\cap\dz{A}\big)\subseteq \mathcal{M}$ or, equivalently $P_\mathcal{M}AP_\mathcal{M}=AP_\mathcal{M}$, where $P_\mathcal{M}$ is the orthogonal projection onto $\mathcal{M}$. If $P_\mathcal{M} \dz{A}\subseteq \dz{A}$ and both $\mathcal{M}$ and $\mathcal{M}^\perp$ are invariant for $A$, then $\mathcal{M}$ is said to be {\em reducing} for $A$; this is equivalent to the following 
\begin{align*}%\label{zabierzow05}
P_{\mathcal{M}}A\subseteq AP_{\mathcal{M}}.  
\end{align*} 
%Set
%\begin{align*}
%\dzn{n}{A}=\bigcap_{k=1}^n \dz{A^k},\quad n\in\nbb\cup\{\infty\}.
%\end{align*}
%Members of $\dzn{\infty}{A}$ are called {\em $C^\infty$-vectors} of $A$. 
If $A$ is closed and densely defined, there exists a (unique) partial isometry $U$ such that $\jd{U}=\jd{A}$ and $A=U|A|$, where $|A|$ is the square root of $A^*A$; the operator $U$ is called the {\em phase} of $A$, $|A|$ is the {\em modulus} of $A$, and $A=U|A|$ is the {\em polar decomposition} of $A$. 

Let $A\in\opr{\hh}$ be selfadjoint and $E_A$ be its spectral measure. We say that $B \in \ogr{\hh}$ {\em spectrally commutes} with $A$ if and only if
\begin{align*}
B E_A(\sigma) = E_A(\sigma) B,\quad \sigma \in \borel{\rbb}. 
\end{align*}
It is known (see, e.g., \cite[Proposition 5.15]{schmudgen}) that $B$ spectrally commutes with $A$ if and only if $B A \subseteq A B$. Furthermore, two selfadjoint operators $A_1$ and $A_2$ in $\hh$ are said to {\em spectrally commute} if and only if their spectral measures mutually commute, i.e., 
\begin{align*}
E_{A_1}(\sigma) E_{A_2}(\omega) = E_{A_2}(\omega) E_{A_1}(\sigma),\quad \sigma, \omega \in \borel{\rbb}
\end{align*}
or, equivalently, $E_{|A_1|}(\sigma) E_{|A_2|}(\omega) = E_{|A_2|}(\omega) E_{|A_1|}(\sigma)$ for all $\sigma, \omega \in \borel{\rbb}$.

Throughout the paper we assume that $(X,\ascr, \mu)$ is a $\sigma$-finite measure space, $\phi$ is an $\ascr$-measurable {\em transformation} of $X$ (i.e., an $\ascr$-measurable mapping $\phi\colon X\to X$) and $w\colon X\to\cbb$ is $\ascr$-measurable.

Let $\mu_w$ and $\mu_w\circ\phi^{-1}$ be measures on $\ascr$ defined by $\mu_w(\sigma)=\int_\sigma|w|^2\D\mu$ and $\mu_w\circ \phi^{-1}(\sigma)=\mu_w\big(\phi^{-1}(\sigma)\big)$. Assume that $\mu_w\circ\phi^{-1}$ is absolutely continuous with respect to  $\mu$. Then the operator 
\begin{align*}
\cfw \colon L^2(\mu) \supseteq \dz{\cfw} \to L^2(\mu)    
\end{align*}
given by
\begin{align*}
\dz{\cfw} = \{f \in L^2(\mu) \colon w \cdot (f\circ \phi) \in L^2(\mu)\},\quad
\cfw f  = w \cdot (f\circ \phi), \quad f \in \dz{\cfw},
\end{align*}
is well defined (see \cite[Proposition 7]{2018-lnim-budzynski-jablonski-jung-stochel}); here, as usual, $L^2(\mu)$ stands for the complex Hilbert space of all square $\mu$-integrable $\ascr$-measurable complex functions on $X$ (with the standard $L^2$ inner product). $\cfw$ is called a {\em weighted composition operator} (abbreviated to {\em wco} throughout the paper). By the Radon-Nikodym theorem (cf.\ \cite[Theorem 2.2.1]{ash}), there exists a unique (up to a set of $\mu$-measure zero) $\ascr$-measurable function $\hfw\colon X \to \rbop$ such that
 \begin{align*}
\mu_w \circ \phi^{-1}(\varDelta) = \int_{\varDelta} \hfw \D \mu, \quad
\varDelta \in \ascr.
\end{align*}
As it follows from \cite[Theorem 1.6.12]{ash} and \cite[Theorem 1.29]{rud}, for every $\ascr$-measurable function $f \colon X \to \rbop$ (or for every $\ascr$-measurable function $f\colon X \to \cbb$ such that $f\circ \phi \in L^1(\mu_w)$), we have
\begin{align*}
\int_X f \circ \phi \D\mu_w = \int_X f \, \hfw \D \mu.
\end{align*}
In particular, this implies that $\dz{\cfw}=L^2\big((1+\hfw)\D\mu\big)$ (see \cite[Proposition 10]{2018-lnim-budzynski-jablonski-jung-stochel}). Therefore, 
\begin{align*}%\label{zdolny01}
\text{$\cfw$ is densely defined if and only if $\hfw < \infty$ a.e. $[\mu]$.}
\end{align*}
Furthermore, the boundedness of $\cfw$ is fully characterized by its Radon-Nikodym derivative
\begin{align}\label{boundedness}
\text{$\cfw \in \ogr{L^2(\mu)}$ if and only if } \hfw \in L^\infty(\mu).
\end{align}

Assume that $\cfw$ is densely defined. Then, for a given $\ascr$-measurable function $f\colon X\to \rbop$ (or $f\colon X\to \cbb$ such that $f\in L^p(\mu_w)$ with some $1\leqslant p<\infty$), one may consider $\efw(f)$, the conditional expectation of $f$ with respect to the $\sigma$-algebra $\phi^{-1}(\ascr)$ and the measure $\mu_w$ (see \cite[Section 2.4]{2018-lnim-budzynski-jablonski-jung-stochel}). Also, there is a unique (up to sets of $\mu$-measure zero) $\ascr$-measurable function $g$ on $X$ such that $g=0$ a.e. $[\mu]$ on $\{\hfw=0\}$ and $\efw(f)=g\circ\phi$ a.e. $[\mu_w]$. We will denote this function by $\efw(f)\circ\phi^{-1}$. Recall that $\efw$ can be regarded as a linear contraction on $L^p(\mu_w)$, $p\in[1,\infty]$, which leaves invariant the convex cone $L^p_+(\mu_w)$ of all $\rbb_+$-valued members of $L^2(\mu_w)$; moreover, $\efw$ is an orthogonal projection on $L^2(\mu_w)$.

Considering $w\equiv 1$ leads to a large subclass of weighted composition operators, the so-called composition operators. Namely, assuming that $\phi\colon X\to X$ is $\ascr$-measurable and satisfies $\mu\circ\phi^{-1}\ll\mu$, the operator $C_{\phi, 1}$ is well defined. We call it the {\em composition operator} induced by $\phi$; for brevity we use the notation $C_\phi:=C_{\phi,1}$. The corresponding Radon-Nikodym derivative $\mathsf{h}_{\phi,1}$ and the conditional expectation $\esf_{\phi,1}$ (see the definition below) are denoted by $\mathsf{h}_\phi$ and $\esf_\phi$, respectively. We caution the reader that $C_\phi$ might not be well defined even if $\cfw$ is an isometry. In particular, if this is the case, $\hsf_\phi$ and $\esf_{\phi}$ do not exist whereas $\hfw$ and $\efw$ do (see \cite[pg. 71]{2018-lnim-budzynski-jablonski-jung-stochel}).

Given an $\ascr$-measurable function $g\colon X\to \cbb$, $M_g$ denotes the operator of multiplication by $g$ acting in $L^2(\mu)$. It is clear, that $M_g=\cfw$ with $\phi=\mathrm{id},$ the identity transformation of $X$, and $w=g$

For any $\ascr$-measurable transformation $\psi$ and any $\ascr$-measurable function $v\colon X\to\cbb$, we define the cocycle products of $v$ along $\psi$ as
\begin{align}\label{jana01}
    p_{\psi, v}^{[0]} = \chi_X \quad \text{and}\quad p_{\psi, v}^{[n]} = \prod_{j=0}^{n-1} v\circ\psi^j \text{ for }n\in\nbb.
\end{align}
In particular, $p_{\phi,w}^{[1]}=w$. It is known (see \cite[Lemma 26]{2018-lnim-budzynski-jablonski-jung-stochel}) that for every $n\in\nbb$, $C_{\phi^n, p_{\phi,w}^{[n]}}$ is well defined and $\cfw^n\subseteq C_{\phi^n, p_{\phi,w}^{[n]}}$. 
For brevity, whenever this leads to no confusion, we use the following 
\begin{align*}%\label{trs01}
    \mu_n = \mu_{p_{\phi,w}^{[n]}}, \quad \hsf_n=\hsf_{\phi^n, p_{\phi,w}^{[n]}},\quad\esf_n=\esf_{\phi^n, p_{\phi,w}^{[n]}},\quad n\in\nbb
\end{align*}
(recall that $\esf_{n}$ exists whenever $\hsf_n<\infty$ a.e. $[\mu]$) and
\begin{equation*}%\label{iggi}
\begin{aligned}
   \mno{n} f&= \hsf_n f,\quad f\in L^2\big((1+\hsf_n^2)\D\mu\big),\\
\pfwn{n} f&= p_{\phi,w}^{[n]} \esf_{n}\big(f_{p_{\phi,w}^{[n]}}\big),\quad f\in L^2(\mu),
\end{aligned} \qquad \quad n\in\nbb,
\end{equation*}
where $f_{v}=\chi_{\{v\neq 0\}}\frac{f}{v}$. Also, we often (especially in proofs) write $w_k = p_{\phi,w}^{[k]}$ for $k \in \nbb$. It is known (see \cite[Lemma 4.2]{2020-mn-benhida-budzynski-trepkowski}) that $\pfwn{n}$ is an orthogonal projection onto $\overline{\ob{\cfw^n}}$. If $C_{\psi, v}$ is densely defined, we define
\begin{align*}
    u_{\psi, v} = \frac{v}{\sqrt{\hsf_{\psi, v} \circ \psi}}.
\end{align*}
In particular, $u_{\phi, w} = \frac{w}{\sqrt{\hsfn{1}\circ\phi}}$. In view of \cite[Theorem 18]{2018-lnim-budzynski-jablonski-jung-stochel}, assuming $\cfw$ is densely defined,
\begin{align*}%\label{zgon01}
    \cfw= C_{\phi,u_{\phi,w}} M_{\sqrt{\hsfn{1}}}
\end{align*}
is the polar decomposition of $\cfw$; in other words, the phase of $\cfw$ is equal to $C_{\phi,u_{\phi,w}}$ and the modulus of $\cfw$ is equal to $M_{\sqrt{\hsfn{1}}}$. The $n$-th power of the phase of $\cfw$ is a composition operator with symbol $\phi^n$ and weight $p_{\phi, u_{\phi,w}}^{[n]}$ (see. \cite[Lem. 26 and 44]{2018-lnim-budzynski-jablonski-jung-stochel}), which is given by
\begin{align*}%\label{djo01}
    p_{\phi, u_{\phi,w}}^{[n]} = \prod_{j=0}^{n-1} u_{\phi,w}\circ\phi^j = \frac{w}{\sqrt{\hsfn{1}\circ\phi}}\cdot\frac{w\circ \phi}{\sqrt{\hsfn{1}\circ \phi^2}}\cdots\frac{w\circ\phi^{n-1}}{\sqrt{\hsfn{1}\circ\phi^n}},\quad n\in\nbb.
\end{align*}
Conversely, the phase of the $n$-th power $\cfw^n$, assuming $\cfw^n$ is densely defined and closed, has the symbol $\phi^n$ and the weight
\begin{align*}%\label{djo02}
    u_{\phi^n, p_{\phi,w}^{[n]}} = \frac{p_{\phi,w}^{[n]}}{\sqrt{\hsfn{n}\circ\phi^n}},\quad n\in\nbb.
\end{align*}
We alert the reader that our notation departs from \cite{2018-lnim-budzynski-jablonski-jung-stochel} to accommodate varying transformations and weights.

\section{Spectrally $n$-centered wco's -- an invitation}

We begin our considerations with a simple example of a composition operator over a discrete measure space. It leads us to properties filling the gap between weak centeredness and centeredness and being related to invariant subspace problem. That the operator posses a nontrivial invariant space is easily seen to be true and requires no reference to centeredness or hyponormality. Nevertheless, we find considering it from a centeredness perspective fruitful. The example and the result that follows (see Proposition \ref{abs_inv}) are based on the same idea that is behind the Campbell's criterion \cite[Th. 2]{1972-pams-campbell}.

\begin{exa}\label{inv_tree}
Let $X$ be countable, $\ascr=2^X$, and $\mu$ be the counting measure. Let $\phi$ be a (measurable) transformation of $X$ such that $\kappa(x) = \card{\phi^{-1}(\{x\})}$, the {\em valency} of $x$, is bounded on $X$. Finally, let $C_\phi$ be the composition operator induced by $\phi$. In view of \cite[Prop. 79]{2018-lnim-budzynski-jablonski-jung-stochel} and \eqref{boundedness}, $C_\phi$ is a well defined bounded operator; moreover, 
\begin{align*}
\hsf_{\phi}(x)=\kappa(x),\quad x\in X.
\end{align*}
Let $\{e_x\}_{x \in X}$ be the standard orthonormal basis of $L^2(\mu)$. Using \cite[Th. 18]{2018-lnim-budzynski-jablonski-jung-stochel} we get
\begin{align*}
C_\phi^* C_\phi e_x = \kappa(x) e_x,\quad C_\phi C_\phi^* e_x = \sum_{y \in \phi^{-1}(\{\phi(x)\})} e_y,\quad x\in X.
\end{align*}
The operator $C_\phi$ is weakly centered  if and only if the following holds (see \cite[Th. 1]{2025-arxiv-budzynski})
\begin{align}\label{weak-1}
\text{the valency $\kappa$ is constant on sets of the form $\phi^{-1}(\{x\})$, $x\in X$.}
\end{align}
Assuming that \eqref{weak-1} holds, we may denote by $c_x$ the value of valency on $\phi^{-1}(\{x\})$. Suppose that $C_\phi$ is hyponormal, i.e., $C_\phi^* C_\phi \Ge C_\phi C_\phi^*$. This, in view of \cite[Th. 53]{2018-lnim-budzynski-jablonski-jung-stochel}, implies 
\begin{align}\label{hyp}
c_x \Ge \kappa(x),\quad x \in \phi(X). 
\end{align}
For a given $\lambda > 0$, we define
\begin{align}\label{minv}
    \mcal{M}_\lambda = \bigvee \{e_x \colon \kappa(x) \Ge \lambda\},
\end{align}
the closure of the linear span of $\{e_x \colon \kappa(x) \Ge \lambda\}$. In view of \eqref{hyp}, if $x\in X$ satisfies $\kappa(x) \Ge \lambda$, then every $y \in \phi^{-1}(\{x\})$ satisfies $\kappa(y) = c_x \Ge \lambda$. This means that $\mcal{M}_\lambda$ is invariant under $C_\phi$. If $C_\phi^* C_\phi$ is not a scalar multiple of the identity (i.e., the valency $\kappa$ is not globally constant), there is $\lambda > 0$ such that $\varnothing \neq \{x \in X \colon \kappa(x) \Ge \lambda\}\neq X$. For such a $\lambda$, the subspace $\mcal{M}_\lambda $ is nontrivial. If the valency $\kappa$ is constant, then $C_\phi$ is (up to a constant) an isometry, which means it has plenty of invariant subspaces.

In view of \cite[Th. 6]{2026-arxiv-budzynski-01}, $C_\phi$ is centered if and only if 
\begin{align}\label{cent}
\text{the valency $\kappa$ is constant on sets of the form $\phi^{-k}(\{x\})$, $x\in X$, $k\in\nbb$.}
\end{align}
A natural way of weakening the above condition is to assume that for a fixed $n\in\nbb\setminus\{1\}$, $C_\phi$ satisfies
\begin{align}\label{weak-n}
\text{the valency $\kappa$ is constant on sets of the form $\phi^{-k}(\{x\})$, $x\in X$, $k\in J_n$}.
\end{align}
Clearly, \eqref{cent} implies \eqref{weak-n}, and \eqref{weak-n} implies \eqref{weak-1}. In particular, $C_\phi$ is weakly centered. Let $c_x^{(n)}$ denote the constant value of a valency $\kappa$ on $\phi^{-n}(\{x\})$. Assume that
\begin{align}\label{hyp-n}
    c_x^{(n)} \Ge \kappa(x), \quad x \in \phi^n(X).
\end{align}
One can see that \eqref{hyp} implies \eqref{hyp-n}. Indeed, let us assume that $c_v \Ge \kappa(v)$ holds for every $v \in \phi(X)$. For a fixed $x \in \phi^n(X)$ we choose arbitrary $z \in \phi^{-n}(\{x\})$. Set $x = x_0$, $x_n = z$, and define recursively $x_{k-1}=\phi(x_k)$ for $1 \Le k \Le n$. Applying \eqref{hyp} yields
\begin{align*}
    \kappa(z)=\kappa(x_n) \Ge \kappa(x_{n-1}) \Ge \dots \Ge \kappa(x_1) \Ge \kappa(x_0)=\kappa(x).
\end{align*}
Consequently, under the assumption of \eqref{weak-n}, we have $c_x^{(n)} \Ge \kappa(x)$, which proves that \eqref{hyp} implies \eqref{hyp-n}. It can easily be shown that the other implication is not true in general. The subspace $\mcal{M}_\lambda$, defined in \eqref{minv}, need not to be invariant under $C_\phi$. However, it is invariant under $C_\phi^n$. Consequently, 
\begin{align}\label{rm00}
    \mcal{M}= \bigvee_{k=1}^{n-1}C_\phi^k\mcal{M}_\lambda
\end{align}
is an invariant subspace for $C_\phi$.

The space $\mcal{M}$ defined in \eqref{rm00} may be trivial even for $n=2$. Indeed, let $X$ be the set of vertexes of a rootless directed tree, partitioned into generations $X_m$, $m \in \zbb$. Assume that the the tree hast the following property: {\em every vertex in $X_{2m}$ has exactly one child and every vertex in $X_{2m+1}$ has exactly two children, $m\in\zbb$}. In other words, the valency is given by
\begin{align*}
    \kappa(x) = 
    \begin{cases} 
        1, & x \in X_{2m}, \\ 
        2, & x \in X_{2m+1}. 
    \end{cases}
\end{align*}
Let $\phi$ be given by the $\paa$ parent function on the tree. Clearly, \eqref{weak-n} and \eqref{hyp-n} with $n=2$ are satisfied. We see that
\begin{align*}
    \mcal{M}_2 &= \bigvee \{e_x \colon \kappa(x) \Ge 2\} = \bigvee_{m \in \zbb} \{e_x \colon x \in X_{2m+1}\},\\
    C_\phi \mcal{M}_2 &= \bigvee_{m \in \zbb} \{e_y \colon y \in X_{2m}\}.
\end{align*}
Consequently, $\mcal{M} = \mcal{M}_2 \vee C_\phi \mcal{M}_2= L^2(\mu)$. 

To ensure that $\mcal{M}$ is not trivial, one can assume e.g., that there exist $\lambda > 0$ and $y_0 \in X$ such that 
\begin{align*}
    \{x \in X \colon \kappa(x) \Ge \lambda\}\neq \varnothing\quad \text{and}\quad \max_{0 \Le k < n} \kappa(\phi^k(y_0)) < \lambda.
\end{align*}
Under this condition, $e_{y_0}$ is orthogonal to $\mcal{M}$, guaranteeing that $\mcal{M} \neq L^2(\mu)$.

As pointed out earlier, even though $\mcal{M}$ may fail to be nontrivial, the operator $C_\phi$ still possesses nontrivial invariant subspaces. For example, for any fixed $k \in \zbb$, the subspace
\begin{align*}
    \mcal{H}_k = \bigvee \{ e_x \colon x \in X_m, m \Ge k \}
\end{align*}
is such a subspace, since $C_\phi \mcal{H}_k \subseteq \mcal{H}_{k+1} \subsetneq \mcal{H}_k$. More generally, one can easily prove that if the graph of $\phi$ is a directed tree (or contains a directed tree as a subtree), then any subspace generated by vectors of the standard orthonormal basis attached to a proper subtree is invariant.

\end{exa}

The conditions \eqref{weak-n} and \eqref{hyp-n} of Example \ref{inv_tree} have natural operator-theoretic counterparts. Indeed, \eqref{weak-n} can be written for a general $\cfw\in\ogr{L^2(\mu)}$ as
\begin{align*}
\hsfn{1}=\efwn{n}(\hsfn{1})\quad \text{a.e. }[\mu_n],
\end{align*}
which is equivalent to commutativity of $\cfw^*\cfw$ and the projection onto $\overline{\ob{\cfw^n}}$ (cf. Lemma \ref{gary01}). This, in turn, can be rewritten as commutativity of $\cfw^*\cfw$ with $|\cfw^n|$ and $V_n \cfw^*\cfw V_n^*$, where $V_n$ is the phase of $\cfw^n$. On the other hand, \eqref{hyp-n} leads to inequality $\cfw^*\cfw\geqslant V_n \cfw^*\cfw V_n^*$.

\begin{pro}\label{abs_inv}
Let $n\in\nbb$, $\lambda > 0$, and $T \in \mathcal{B}(\mathcal{H})$. Let $V_n$ denotes the phase of $T^n$. Set $B = T^*T$ and $C_n = V_n B V_n^*$. Assume that $T$ satisfies the following conditions:
\begin{enumerate}
    \item[(i)] $B$ commutes with $|T^n|$ and $C_n$,
    \item[(ii)] $B \geqslant C_n$.
\end{enumerate}
Let $\mcal{M}_{\text{inv}} = \bigvee_{k=0}^{n-1} T^k \mcal{M}_\lambda$, where $\mcal{M}_\lambda = \ob{E_B([\lambda, \infty))}$. Then $\mcal{M}_{\text{inv}}$ is invariant under $T$.
\end{pro}

\begin{proof}
We first show that $\mcal{M}_\lambda$ is invariant for $T^n$. The product $V_n^* V_n$ is the orthogonal projection onto $\overline{\ob{|T^n|}}$. Since $B$ commutes with $|T^n|$, it commutes with $V_n^* V_n$ as well. Consequently, we have
\begin{align*}
    C_n T^n &= (V_n B V_n^*) V_n |T^n| = V_n B (V_n^* V_n) |T^n| = V_n (V_n^* V_n) B |T^n| = V_n |T^n| B = T^n B.
\end{align*}
This yields $E_{C_n}(\Delta) T^n = T^n E_B(\Delta)$ for every $\Delta\in\borel{\cbb}$. In particular, we obtain
\begin{align*}
    E_{C_n}([\lambda, \infty)) T^n = T^n E_B([\lambda, \infty)).
\end{align*}
Thus $T^n x = T^n E_B([\lambda, \infty)) x = E_{C_n}([\lambda, \infty)) T^n x$ for every $x\in \mcal{M}_\lambda$, implying
\begin{align}\label{neg}
T^n \mcal{M}_\lambda \in \ob{E_{C_n}([\lambda, \infty))}.
\end{align}
Since $B$ and $C_n$ commute, and $C_n\leq B$, we get $E_{C_n}([\lambda, \infty)) \Le E_B([\lambda, \infty))$. Therefore, 
\begin{align*}
\ob{E_{C_n}([\lambda, \infty))} \subseteq\ob{E_B([\lambda, \infty))}.    
\end{align*}
Combining this and \eqref{neg} we see that $\mcal{M}_\lambda$ is invariant for $T^n$.

That $\mcal{M}_{\text{inv}}$ is invariant for $T$ follows from
\begin{align*}
    T \mcal{M}_{\text{inv}} = T \left( \bigvee_{k=0}^{n-1} T^k \mcal{M}_\lambda \right) \subseteq \bigvee_{k=1}^n T^k \mcal{M}_\lambda = \bigvee_{k=1}^{n-1} T^k \mcal{M}_\lambda \vee T^n \mcal{M}_\lambda,
\end{align*}
and the fact that $T^n \mcal{M}_\lambda \subseteq \mcal{M}_\lambda$. This completes the proof.
\end{proof}

Commutativity of $\cfw^*\cfw$ and the projection onto $\overline{\ob{\cfw^n}}$ is equivalent to commutativity of the spectral measures $E_{\cfw^*\cfw}$ and $E_{\cfw^n\cfw^{n*}}$, thus Proposition \ref{abs_inv} suggests studying this property. Since both weak centeredness and centeredness are closed under taking the adjoints, considering commutativity of $E_{\cfw^{n*}\cfw^n}$ and $E_{\cfw^n\cfw^{n*}}$ seems even more natural.

\begin{dfn}
Let $\hh$ be a Hilbert space, $n\in\nbb$, and $A\in {\mathfrak{C}}_{n}$. For $k\in J_n$, let $E_{1,k}=E_{A^{k*}A^k}$ and $E_{2,k}=E_{A^k A^{k*}}$ be the spectral measures of $A^{k*}A^k$ and $A^k A^{k*}$, respectively. We say that $A$ is {\em spectrally $n$-centered} if and only if
\begin{align*}%\label{zabierzow01}
E_{i,k}(\sigma)E_{j,l}(\omega) = E_{j,l}(\omega)E_{i,k}(\sigma), \quad i,j \in \{1,2\}, \ k,l \in J_n, \ \sigma,\omega \in \borel{\rbb};
\end{align*}
$A$ is said to be {\em spectrally weakly $n$-centered} if and only if
\begin{align*}%\label{lato01}
E_{1,1}(\sigma)E_{2,k}(\omega)=E_{2,k}(\omega)E_{1,1}(\sigma), \quad  k \in J_n,\ \sigma,\omega\in\borel{\rbb}.
\end{align*}
\end{dfn}

For $A\in\ogr{\hh}$, we simply say that $A$ is {\em $n$-centered}; $1$-centeredness is weak centeredness. If $A\in {\mathfrak{C}}_{\infty}$ and $A$ is spectrally $n$-centered for every $n\in\nbb$, then $A$ is called {\em spectrally centered}. For $A\in\ogr{\hh}$, this is equivalent to the centeredness in the sense of Morrel.    

A notion related to centerednes is half-centeredness, introduced in \cite{2018-oam-giselsson} by Giselsson. He showed that a half-centered $T\in\ogr{\hh}$ is centered if and only if $T^{k*}T^k\jd{T^*}\subseteq \jd{T^*}$ (see \cite[Proposition 2.1]{2018-oam-giselsson}). We generalized this concept to the unbounded scenario in \cite{2026-arxiv-budzynski-01} in the following way: $A\in {\mathfrak{C}}_{n}$ is {\em spectrally half-centered} if and only if
\begin{align*}%\label{rzaska01}
E_{1,n}(\sigma)E_{1,m}(\omega)=E_{1,m}(\omega)E_{1,n}(\sigma),\quad \sigma,\omega\in\borel{\cbb},\ n,m \in \nbb.
\end{align*}
Clearly, for $A\in\ogr{\hh}$, $A$ is half-centered if and only if $A$ is spectrally half-centered. Below, we provide a criterion for spectral centeredness inspired by the Giselsson's result.

\begin{pro}
Let $T$ be spectrally centered. Then $\jd{T^*}$ reduces $T^{k*}T^k$ for every $k\in\nbb$.
\end{pro}
\begin{proof}
Fix $k\in\nbb$. Since $T$ is centered, we have
\begin{align*}%\label{zabierzow04}
E_{1,k}(\sigma)E_{2,1}\big(\{0\}\big)=E_{2,1}\big(\{0\}\big)E_{1,k}(\sigma),\quad \sigma\in\borel{\rbb}.
\end{align*}
This implies that
\begin{align*}%\label{zabierzow05}
PT^{k*}T^k\subseteq T^{k*}T^k P,\quad \sigma\in\borel{\rbb},
\end{align*}
with $P=E_{2,1}\big(\{0\}\big)$. Now it suffices to observe that $P$ is a projection onto $\jd{TT^*}=\jd{T^*}$.
\end{proof}

\section{Spectrally $n$-centered wco's -- a characterisation}

Let $n\in\nbb$ and $\cfw\in\cfrak{n}$. In view of \cite[Lem. 43 and 44]{2018-lnim-budzynski-jablonski-jung-stochel}, $\hsfn{k}<\infty$ a.e. $[\mu]$, $\efwn{k}$ is well defined, $\cfwn{k}$ is densely defined and $\cfw^k=C_{\phi^k, w_k}$ for all $k\in J_n$. Let $E_{1,k}$ and $E_{2,k}$ denote the spectral measures of $\cfw^{k*}\cfw^k$ and $\cfw^k \cfw^{k*}$, $k\in J_n$, respectively. In view of \cite[Ex. 5.3, pg. 93]{schmudgen} and \cite[Eq. (19)]{2025-arxiv-budzynski} we have
\begin{equation}\label{jan02}
\begin{aligned}
   E_{1,k}(\sigma) &= M_{\chi_{\varOmega_{\sigma,k}}} \\
   E_{2,k}(\sigma) &= M_{\chi_{\varGamma_{\sigma,k}}} \psf_k + \chi_\sigma(0)(I-\psf_k),
\end{aligned} \qquad \sigma\in\borel{\rbb}
\end{equation}
where 
\begin{align}\label{trawa}
\varOmega_{\sigma,k}:=\big({\hsfn{k}}\big)^{-1}(\sigma)\quad \text{and}\quad\varGamma_{\sigma,k}:=(\hsfn{k}\circ\phi^k)^{-1}(\sigma).
\end{align}

We begin by showing that the problem of spectral $n$-centeredness of a {\em wco} can be reduced to the case when the weight function is non-negative.
\begin{lem}\label{jana-global}
Let $n \in \nbb$ and $\cfw \in \cfrak{n}$. Then $\cfw$ is spectrally $n$-centered if and only if $C_{\phi, |w|}$ is spectrally $n$-centered.
\end{lem}
\begin{proof}
For the sake of readability, we write $w_k = p_{\phi,w}^{[k]}$.

Let $s \colon X \to \cbb$ be an $\ascr$-measurable function defined by
\begin{equation}\label{ess}
    s(x)=\left\{ \begin{array}{cc}
         1&  \text{if } w(x)=0,\\
         \frac{w(x)}{|w(x)|}& \text{if } w(x)\neq0. 
    \end{array}\right.
\end{equation}
By \cite[Prop. 115]{2018-lnim-budzynski-jablonski-jung-stochel}, for any $k\in J_n$, $C_{\phi^k, |w_k|}$ is densely defined and 
\begin{align}\label{eq:phase_decomp}
    \cfw^k = M_{s_k} C_{\phi^k, |w_k|},\quad k\in J_n,
\end{align}
with $s_k=p_{\phi, s}^{[k]}$ defined via \eqref{jana01}; moreover, 
\begin{align}\label{jana04}
\mu_{w_k}=\mu_{|w_k|},\quad \hsf_k=\hsf_{\phi^k,|w_k|},\quad \efwn{k} =\esf_{\phi^k, |w_k|}.
\end{align}
Clearly, $|s_k| = 1$ a.e. $[\mu]$, meaning the multiplication operator $M_{s_k}$ is unitary.

For $k\in J_n$, let $\widetilde{E}_{1,k}$ and $\widetilde{E}_{2,k}$ denote the respective spectral measures of $C_{\phi, |w|}^{k*}C_{\phi, |w|}^k$ and $C_{\phi, |w|}^k C_{\phi, |w|}^{k*}$. By \eqref{eq:phase_decomp}, we have
\begin{align*}
    \cfw^{k*}\cfw^k = C_{\phi^k, |w_k|}^* M_{\bar{s}_k} M_{s_k} C_{\phi^k, |w_k|} = C_{\phi^k, |w_k|}^* C_{\phi^k, |w_k|}=C_{\phi, |w|}^{k*}C_{\phi, |w|}^k.
\end{align*}
Therefore, we get (see \eqref{jan02})
\begin{align}\label{jana02}
\widetilde{E}_{1,k}(\sigma)=E_{1,k}(\sigma)=M_{\chi_{\varOmega_{\sigma, k}}}, \quad \sigma\in\borel{\cbb}
\end{align}
Since $\cfw^k\cfw^{k*} = M_{s_k} C_{\phi^k, |w_k|} C_{\phi^k, |w_k|}^* M_{\bar{s}_k}$, we consequently have 
\begin{align}\label{jana03}
    E_{2,k}(\sigma) = M_{s_k} \widetilde{E}_{2,k}(\sigma) M_{\bar{s}_k}, \quad \sigma \in \borel{\cbb}.
\end{align}
By definition, $\cfw$ is spectrally $n$-centered if and only if $\{E_{i,k}\}_{i\in\{1,2\}, k\in J_n}$ is commutative. We will show that this holds if and only if $\{\widetilde{E}_{i,k}\}_{i\in\{1,2\}, k\in J_n}$ is mutually commutative.

Let $k, m\in J_n$. The case of $E_{1,k}$ and $E_{1,m}$ is clear in view of \eqref{jana02}. For the case of $E_{1,k}$ and $E_{2,m}$ we use \eqref{jana02} again and \eqref{jana03} to get\allowdisplaybreaks
\begin{align*}
    E_{1,k}(\sigma) E_{2,m}(\omega) 
    &= M_{\chi_{\varOmega_{\sigma,k}}} M_{s_m} \widetilde{E}_{2,m}(\omega) M_{\bar{s}_m}
    = M_{s_m}M_{\chi_{\varOmega_{\sigma,k}}} \widetilde{E}_{2,m}(\omega) M_{\bar{s}_m}\\
    &=M_{s_m} \widetilde{E}_{1,k}(\sigma) \widetilde{E}_{2,m}(\omega) M_{\bar{s}_m},\\
    E_{2,m}(\omega) E_{1,k}(\sigma) 
    &= M_{s_m} \widetilde{E}_{2,m}(\omega) M_{\bar{s}_m} M_{\chi_{\varOmega_{\sigma,k}}}
    = M_{s_m} \widetilde{E}_{2,m}(\omega) M_{\chi_{\varOmega_{\sigma,k}}} M_{\bar{s}_m} \\
    &= M_{s_m} \widetilde{E}_{2,m}(\omega) \widetilde{E}_{1,k}(\sigma) M_{\bar{s}_m}.
\end{align*}
Since $M_{s_m}$ is unitary, $E_{1,k}(\sigma)$ commutes with $E_{2,m}(\omega)$ if and only if $\widetilde{E}_{1,k}(\sigma)$ commutes with $\widetilde{E}_{2,m}(\omega)$. Now we consider $E_{2,k}$ and $E_{2,m}$. We may assume, without loss of generality, that $m \geqslant k$. By definition, $s_m = s_k \cdot (s_{m-k} \circ \phi^k)$. Let $g = s_{m-k} \circ \phi^k$. Then $M_g\in\ogr{L^2(\mu)}$ and $M_{\bar{s}_k} M_{s_m} = M_g$. Obviously, $g$ is measurable with respect to $\phi^{-k}(\ascr)$. In view of \eqref{jan02} and \eqref{jana04}, 
\begin{align*}
\widetilde{E}_{2,k}(\sigma) = M_{\chi_{\varGamma_{\sigma,k}}} \widetilde{\psf}_k + \chi_\sigma(0)(I-\widetilde{\psf}_k),
\end{align*}
where $\widetilde{\psf}_k = \psf_{\phi^k, |w_k|}$. Because $g$ is $\phi^{-k}(\ascr)$-measurable, we have
\begin{align*}
    \widetilde{\psf}_k M_g f = |w_k| \esf_k \big( g f_{|w_k|} \big) = g |w_k| \esf_k \big( f_{|w_k|} \big) = M_g \widetilde{\psf}_k f,\quad  f \in L^2(\mu).
\end{align*}
This shows that $\widetilde{\psf}_k$ commutes with $M_g$, which immediately implies that $\widetilde{E}_{2,k}(\sigma)$ commutes with $M_g$. By \eqref{jana03} we get\allowdisplaybreaks
\begin{align*}
    E_{2,k}(\sigma) E_{2,m}(\omega) &= M_{s_k} \widetilde{E}_{2,k}(\sigma) M_{\bar{s}_k} M_{s_m} \widetilde{E}_{2,m}(\omega) M_{\bar{s}_m} 
    = M_{s_k} \widetilde{E}_{2,k}(\sigma) M_g \widetilde{E}_{2,m}(\omega) M_{\bar{s}_m} \\
    &= M_{s_k} M_g \widetilde{E}_{2,k}(\sigma) \widetilde{E}_{2,m}(\omega) M_{\bar{s}_m} 
    = M_{s_m} \widetilde{E}_{2,k}(\sigma) \widetilde{E}_{2,m}(\omega) M_{\bar{s}_m}.
\end{align*}
In a similar fashion, using commutativity of $M_{\bar{g}}$ and $\widetilde{E}_{2,k}(\sigma)$, we get
\begin{align*}
    E_{2,m}(\omega) E_{2,k}(\sigma) =M_{s_m} \widetilde{E}_{2,m}(\omega) \widetilde{E}_{2,k}(\sigma) M_{\bar{s}_m}.
\end{align*}
Comparing the two above, we see that $E_{2,k}(\sigma)$ commutes with $E_{2,m}(\omega)$ if and only if $\widetilde{E}_{2,k}(\sigma)$ commutes with $\widetilde{E}_{2,m}(\omega)$. Since the mutually commutative property holds for all pairs in $\{E_{i,k}\}$ if and only if it holds for $\{\widetilde{E}_{i,k}\}$, the proof is complete.
\end{proof}
We showed in \cite[Th. 4]{2026-arxiv-budzynski-01} that the following conditions
\begin{itemize}
\item[(ii)] for every $n\in\nbb$, $\pfwn{1} M_{\hsfn{n}}\subseteq M_{\hsfn{n}} \pfwn{1}$,
\item[(iii)] for every $n\in\nbb$, $\efwn{1}(\hsfn{n})= \hsfn{n}$ a.e. $[\mu_w]$,
\end{itemize}are equivalent whenever $\cfw$ is bounded. This led to a characterisation of (bounded) centered {\em wco}'s. Below we generalise this result and apply it later in the context of spectrally $n$-centered operators.
\begin{lem}\label{gary01}
Let $n\in\nbb$ and $\cfw\in \cfrak{n}$. Let $k, l \in J_n$. Then the following conditions are equivalent:
\begin{itemize}
\item[(i)] $\pfwn{k} \mno{l}\subseteq \mno{l} \pfwn{k}$,
\item[(ii)] $\efwn{k}(\hsf_l)= \hsf_l$ a.e. $[\mu_k]$.
\end{itemize}
\end{lem}
\begin{proof}
We write $w_k = p_{\phi,w}^{[k]}$ and $s_k = p_{\phi,s}^{[k]}$, where $s$ is defined via \eqref{ess}.

We first reduce the proof to the case $w \geqslant 0$ a.e. $[\mu]$. The projection $\pfwn{k}$ satisfies $\pfwn{k} = M_{s_k} \widetilde{\psf}_k M_{\bar{s}_k}$, where $\widetilde{\psf}_k=\psf_{\phi^k, |w_k|}$. Since $|s_k|=1$ a.e. $[\mu]$, $\mno{l}$ commutes with both the operators $M_{s_k}$ and $M_{\bar{s}_k}$. Consequently, $\pfwn{k} \mno{l} \subseteq \mno{l} \pfwn{k}$ if and only if $\widetilde{\psf}_k \mno{l} \subseteq \mno{l} \widetilde{\psf}_k$. Furthermore, condition (ii) relies entirely on $\hsf_l$ and $\mu_k$, both of which are independent of the phase of $w$ due to \eqref{jana04}. Thus, assuming $w \geqslant 0$ a.e. $[\mu]$ leads to no loss in generality. Consequently, $w_k \geqslant 0$ a.e. $[\mu]$.

(i)$\Rightarrow$(ii) Condition (i) implies that
\begin{align*}
w_k\hsf_l\efwn{k} (f_{w_k})= w_k\efwn{k}(\hsf_l f_{w_k}) \text{ a.e. $[\mu]$}, \quad f\in \dz{\mno{l}}. 
\end{align*}
Let $L^2_+(\mu)=\{f\in L^2(\mu) \colon f\geqslant 0\}$ and ${\EuScript D}_+(\mno{l})= L^2_+(\mu)\cap \dz{\mno{l}}$. Clearly, ${\EuScript D}_+(\mno{l})$ is dense in $L^2_+(\mu)$, and we have
\begin{align}\label{gary03}
w_k \hsf_l \efwn{k} (f_{w_k})= w_k\efwn{k}(\hsf_l f_{w_k}) \text{ a.e. $[\mu]$}, \quad f\in {\EuScript D}_+(\mno{l}).
\end{align}
Since all functions in \eqref{gary03} are $\rbb_+$-valued, we may apply the monotone convergence theorem. For any $f\in L^2_+(\mu)$, there exists a sequence $\{f_m\}\subseteq {\EuScript D}_+(\mno{l})$ such that $f_m\nearrow f$ as $m\to +\infty$ (this follows from $\hsf_l<+\infty$ a.e. $[\mu]$). Passing to the limit, we infer that the equality in \eqref{gary03} is satisfied for any $f\in L^2_+(\mu)$. Using standard measure theoretic arguments and $\efwn{k}(\hsf_l g)=\efwn{k}(\hsf_l g \chi_{\{w_k\neq 0\}})$ a.e. $[\mu_k]$, which is valid for every $\ascr$-measurable $g\colon X\to \rbb_+$, we deduce that 
\begin{align*}
\hsf_l \efwn{k} (f)= \efwn{k}(\hsf_l f) \text{ a.e. $[\mu_k]$}
\end{align*}
holds for every $\ascr$-measurable $f\colon X\to\rbb_+$. Setting $f \equiv 1$ yields $\hsf_l = \efwn{k}(\hsf_l)$ a.e. $[\mu_k]$, which proves (ii).

(ii)$\Rightarrow$(i) We show that $\pfwn{k}\dz{\mno{l}}\subseteq \dz{\mno{l}}$. Take $f\in \dz{\mno{l}}=L^2\big((1+\hsf_l^2)\D\mu\big)$. Clearly, $\pfwn{k}f\in L^2(\mu)$. Since $\big|\efwn{k}(f_{w_k})\big|^2\leqslant \efwn{k}\big(|f_{w_k}|^2\big)$ a.e. $[\mu_k]$ and $f\in L^2(\hsf_l^2\D\mu)$, we get the following
\begin{align*}
\int_X |\pfwn{k} f|^2 \hsf_l^2\D\mu&= \int_X |w_k|^2 \big|\efwn{k} (f_{w_k})\big|^2 \hsf_l^2\D\mu\leqslant \int_X |w_k|^2 \efwn{k} \big(|f_{w_k}|^2\big) \hsf_l^2\D\mu\\
&= \int_X \efwn{k} \big(|f_{w_k}|^2\big) \hsf_l^2\D\mu_k.
\end{align*}
Condition (ii) implies that $\efwn{k}(\hsf_l^2) = \hsf_l^2$ a.e. $[\mu_k]$. Using this, we continue\allowdisplaybreaks
\begin{align*}
\int_X \efwn{k} \big(|f_{w_k}|^2\big) \hsf_l^2\D\mu_k &= \int_X \efwn{k} \big(\hsf_l^2|f_{w_k}|^2\big)\D\mu_k = \int_X \hsf_l^2|f_{w_k}|^2\D\mu_k \\
&\leqslant \int_X \hsf_l^2|f|^2\D\mu<+\infty.
\end{align*}
Thus $\pfwn{k} f\in L^2\big((1+\hsf_l^2)\D\mu\big)=\dz{\mno{l}}$. Consequently, $\pfwn{k}\dz{\mno{l}}\subseteq \dz{\mno{l}}$. That the equality $\pfwn{k} \mno{l} f=\mno{l} \pfwn{k} f$ is valid for all $f\in\dz{\mno{l}}$ follows immediately from \cite[(A.13) in Appendix A]{2018-lnim-budzynski-jablonski-jung-stochel} and the fact that $f_{w_k}$ and $\hsf_l f_{w_k}$ belong to $L^2(\mu_k)$. This completes the proof.
\end{proof}
The above brings us to the main result of the paper. Before stating it, we introduce terminology: given a $\sigma$-finite measure space $(X,\ascr, \mu)$, a $\sigma$-algebra $\bscr\subseteq \ascr$, an $\ascr$-measurable function $f\colon X\to \overline{\rbb}_+$, and a set $\varOmega\in \ascr$, we say that $f$ is {\em $\bscr$-measurable relative to $\varOmega$} if and only if there exist a $\bscr$-measurable function $g\colon X\to \overline{\rbb}_+$ such that $f=g$ a.e. $[\mu]$ on $\varOmega$.
\begin{thm}\label{gary02}
Let $n\in\nbb$ and $\cfw\in \cfrak{n}$. Then the following conditions are equivalent:
\begin{itemize}
\item[(i)] $\cfw$ is spectrally $n$-centered,
\item[(ii)] for all $k,l \in J_n$, $\pfwn{k} \mno{l}\subseteq \mno{l} \pfwn{k}$,
\item[(iii)] for all $k,l \in J_n$, $\efwn{k}(\hsf_l)= \hsf_l$ a.e. $[\mu_k]$,
\end{itemize}
\end{thm}
\begin{proof}
Write $w_m = p_{\phi,w}^{[m]}$ for $m \in J_n$. In view of Lemma \ref{jana-global}, we may assume without loss of generality that $w \geqslant 0$ a.e. $[\mu]$. Note that $\pfwn{k}$ is the orthogonal projection onto $\overline{\ob{C_{\phi^k,w_k}}}$ (see \cite[Lemma 4.2]{2020-mn-benhida-budzynski-trepkowski}). Thus, in view of \eqref{jan02}, we have $\pfwn{k} = E_{2,k}\big((0, \infty)\big)$.

(i)$\Rightarrow$(ii) Assume that $\cfw$ is spectrally $n$-centered and fix $k, l \in J_n$. Then the spectral measures $E_{2,k}$ and $E_{1,l}$ commute. In particular, $E_{2,k}\big((0, \infty)\big) = \pfwn{k}$ commutes with $E_{1,l}(\sigma)$, $\sigma \in \borel{\rbb}$. Since $E_{1,l}$ is the spectral measure of $\mno{l}$, $\pfwn{k}$ spectrally commutes with $\mno{l}$. This is equivalent to $\pfwn{k} \mno{l} \subseteq \mno{l} \pfwn{k}$. Since $k, l$ can be chosen arbitrarily, (ii) holds.

(ii)$\Leftrightarrow$(iii) Follows from Lemma \ref{gary01}.

(iii)$\Rightarrow$(i) Assume that (iii) holds. Fix $k,l\in J_n$ and $\sigma,\omega\in\borel{\rbb}$. We show that 
\begin{align}\label{kom01}
E_{i,k}(\sigma)E_{j,l}(\omega)=E_{j,l}(\omega)E_{i,k}(\sigma),\quad i,j\in\{1,2\}.
\end{align}
Since $E_{1,k}(\sigma)$ and $E_{1,l}(\omega)$ are bounded multiplication operators (see \eqref{jan02}), they commute, which proves \eqref{kom01} for $i=j=1$.  By the equivalence of (iii) and (ii), $\pfwn{l}$ spectrally commutes with $\mno{k}$, which implies it commutes with $E_{1,k}(\sigma)$. Furthermore, $E_{1,k}(\sigma)$, being a bounded multiplication operator, commutes with $M_{\chi_{\varGamma_{\omega, l}}}$ (cf. \eqref{trawa}). Using \eqref{jan02}, we deduce the equality in \eqref{kom01} with $i=1$ and $j=2$. For the proof of the equality in \eqref{kom01} with $i=j=2$, it suffices to show that projections $\pfwn{k}, \pfwn{l}$ and multiplication operators $M_{\chi_{\varGamma_{\sigma, k}}}, M_{\chi_{\varGamma_{\omega, l}}}$ are mutually commutative. We may assume without loss of generality that $k > l$. Since $\ob{\cfw^k}\subseteq \ob{\cfw^l}$, $\pfwn{k} \pfwn{l} = \pfwn{l} \pfwn{k} = \pfwn{k}$. Furthermore, $\chi_{\varGamma_{\sigma, k}} = \chi_{ \hsf_k^{-1}(\sigma)}\circ \phi^k $ is measurable with respect to $\phi^{-k}(\ascr)\subseteq \phi^{-l}(\ascr)$, which means that $\chi_{\varGamma_{\sigma, k}}$ is $\phi^{-l}(\ascr)$-measurable. This immediately yields $\pfwn{l} M_{\chi_{\varGamma_{\sigma, k}}} = M_{\chi_{\varGamma_{\sigma, k}}} \pfwn{l}$. It remains to show that $\pfwn{k}$ commutes with $M_{\chi_{\varGamma_{\omega, l}}}$. Since $\varGamma_{\omega, l} = (\hsf_l \circ \phi^l)^{-1}(\omega)$, it suffices to prove that $\hsf_l \circ \phi^l$ is $\phi^{-k}(\ascr)$-measurable relative to $\{w_k \neq 0\}$. By condition (iii), $\efwn{k-l}(\hsf_l) = \hsf_l$ a.e. $[\mu_{k-l}]$. This means $\hsf_l$ is $\phi^{-(k-l)}(\ascr)$-measurable relative to $\{w_{k-l} \neq 0\}$, ensuring there exists an $\ascr$-measurable function $g$ such that 
\begin{align*}
    w_{k-l} \hsf_l = w_{k-l} \cdot g \circ \phi^{k-l} \quad \text{a.e. } [\mu].
\end{align*}
Using \cite[Lem. 5]{2018-lnim-budzynski-jablonski-jung-stochel}, we get
\begin{align*}
    w_k \cdot \hsf_l \circ \phi^l = w_k \cdot g \circ \phi^k \quad \text{a.e. } [\mu],
\end{align*}
which proves that $\hsf_l \circ \phi^l$ is $\phi^{-k}(\ascr)$-measurable relative to $\{w_k \neq 0\}$. Thus, $M_{\chi_{\varGamma_{\omega, l}}}$ commutes with $\pfwn{k}$. Since $k,l\in J_n$ and $\sigma,\omega\in\borel{\rbb}$ were arbitrary, $\big\{E_{i,k}\colon i\in\{1,2\}, k\in J_n\big\}$ is mutually commutative, which completes the proof.
\end{proof}

An example of a $\cfw$ that, for a given $n\in\nbb$ is $n$-centered, but not $(n+1)$-weakly centered can be constructed easily.

\begin{exa}
Let $n \in \nbb$. Let $X = \{0\} \cup \{x_1, \ldots, x_{2n+1}\} \cup \{y_1, \ldots, y_{2n+1}\}$, let $\ascr = 2^X$, and let $\mu$ be the counting measure on $X$. We define a transformation $\phi \colon X \to X$ by 
\begin{align*}
    \phi(0) = 0,\quad \phi(x_1) = \phi(y_1) = 0,
\end{align*}
and 
\begin{align*}
    \phi(x_i) = x_{i-1},\quad \phi(y_i) &= y_{i-1}, \quad i \in \{2, \ldots, 2n+1\},
\end{align*}
and a weight function $w \colon X \to \rbb_+$ by $w(0) = 0$, $w(x_{2n+1}) = \sqrt{2}$, $w(y_{2n+1}) = \sqrt{3}$, and $w(z) = 1$ for all other $z \in X$. The graph illustrating $\phi$ and $w$ is presented in Figure \ref{fig:n_weak_strict} (the arrows dictate the mapping of $\phi$, i.e., $\phi(x) = y$ is represented by an arrow from $x$ to $y$, and the numbers above the vertices denote the values of the weight $w$). Let $\cfw$ be the weighted composition operator induced by $\phi$ and $w$.
\begin{figure}[ht]
\begin{center}
\begin{tikzpicture}[scale=0.9, transform shape]
\tikzstyle{every node} = [circle,fill=gray!30]

\node (0)[font=\footnotesize, inner sep = 3pt, label={[fill=none]above:$0$}] at (0,0) {$0$};

\node (x1)[font=\footnotesize, inner sep = 2pt, label={[fill=none]above:$1$}] at (2,1.0) {$x_1$};
\node (xdots)[fill=none] at (4,1.0) {$\cdots$};
\node (x2n)[font=\footnotesize, inner sep = 1pt, label={[fill=none]above:$1$}] at (6,1.0) {$x_{2n}$};
\node (x2np1)[font=\footnotesize, inner sep = 0pt, label={[fill=none]above:$\sqrt{2}$}] at (8.5,1.0) {$x_{2n+1}$};

\node (y1)[font=\footnotesize, inner sep = 2pt, label={[fill=none]above:$1$}] at (2,-1.0) {$y_1$};
\node (ydots)[fill=none] at (4,-1.0) {$\cdots$};
\node (y2n)[font=\footnotesize, inner sep = 1pt, label={[fill=none]above:$1$}] at (6,-1.0) {$y_{2n}$};
\node (y2np1)[font=\footnotesize, inner sep = 0pt, label={[fill=none]above:$\sqrt{3}$}] at (8.5,-1.0) {$y_{2n+1}$};

\draw[<-] (0) --(x1);
\draw[<-, dashed] (x1) --(xdots);
\draw[<-, dashed] (xdots) --(x2n);
\draw[<-] (x2n) --(x2np1);

\draw[<-] (0) --(y1);
\draw[<-, dashed] (y1) --(ydots);
\draw[<-, dashed] (ydots) --(y2n);
\draw[<-] (y2n) --(y2np1);

\draw[<-] (0) to[out=135, in=225, looseness=6] (0);
\end{tikzpicture}
\end{center}
\caption{\label{fig:n_weak_strict} The graph illustrating a $\cfw$ that is $n$-centered, but not $(n+1)$-centered.}
\end{figure}
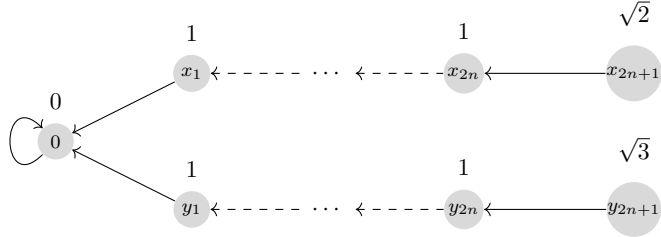
Since $X$ is finite, $\cfw$ is bounded. We first show that $\cfw$ is $n$-centered. Set $w_k = p_{\phi,w}^{[k]}$. Since $w(0) = 0$, we have $w_k(0) = 0$. For $z \in \{x_i, y_i\}$, $w_k(z) \neq 0$ if and only if $i \geqslant k$. The only sets of the form $\phi^{-k}(\{v\})$ that are neither singletons nor empty, are  given by $\phi^{-k}(\{0\}) = \{0\} \cup \{x_1, \ldots, x_k\} \cup \{y_1, \ldots, y_k\}$. Thus, $\efwn{k}(\hsf_l) = \hsf_l$ a.e. $[\mu_k]$ holds if and only if $\hsf_l(x_k) = \hsf_l(y_k)$. Since $k, l \in J_n$, we have $k+l \leqslant 2n < 2n+1$. Thus  $\phi^{-l}(\{x_k\}) = \{x_{k+l}\}$ and $\phi^{-l}(\{y_k\}) = \{y_{k+l}\}$. This yields (see \cite[Eq. (6.5)]{2018-lnim-budzynski-jablonski-jung-stochel})
\begin{align*}
    \hsf_l(x_k) = |w_l(x_{k+l})|^2 = 1 \quad \text{and}\quad 
    \hsf_l(y_k) = |w_l(y_{k+l})|^2 = 1.
\end{align*}
Because $\hsf_l(x_k) = \hsf_l(y_k)$, the condition $\efwn{k}(\hsf_l) = \hsf_l$ a.e. $[\mu_k]$ is satisfied for all $k,l \in J_n$. Thus, $\cfw$ is spectrally $n$-centered. On the other hand, we have
\begin{align*}
    \hsf_n(x_{n+1}) &= |w_n(x_{2n+1})|^2 = (\sqrt{2})^2 = 2, \\
    \hsf_n(y_{n+1}) &= |w_n(y_{2n+1})|^2 = (\sqrt{3})^2 = 3.
\end{align*}
Moreover, $w_{n+1}(x_{n+1}) = w_{n+1}(y_{n+1}) = 1 \neq 0$. Thus, $\hsf_n$ is not constant on $\phi^{-(n+1)}(\{0\})$ relative to $\{w_{n+1}\neq 0\}$. Consequently, $\efwn{n+1}(\hsf_n) = \hsf_n$ a.e. $[\mu_{n+1}]$ fails, meaning $\cfw$ is not $(n+1)$-weakly centered.
\end{exa}

\begin{rem}
An $n$-centered operator which is not $(n+1)$-centered can be constructed over a non-discrete measure space by embedding the discrete example into the real line. Let $X = \rbb_+$, $\ascr=\borel{\rbb_+}$, and $\mu$ be the Lebesgue measure on $\rbb_+$. Given any $\psi \colon \zbb_+ \to \zbb_+$, we define $\phi=\phi_{\psi} \colon X \to X$ by
\begin{align*}
    \phi(x) = x - \lfloor x \rfloor + \psi(\lfloor x \rfloor), \quad x \in X.
\end{align*}
Let $t\in\rbb_+$ and let $\varDelta\in\ascr$ satisfy $\Delta \subseteq [\lfloor t \rfloor, \lfloor t \rfloor + 1)$. Then we have 
\begin{align*}
\phi^{-1}(\Delta) = \bigcup_{k \in \psi^{-1}(\{\lfloor t \rfloor\})} (\Delta - \lfloor t \rfloor + k).
\end{align*}
By translation invariance of the Lebesgue measure, we have
\begin{align*}
\mu(\phi^{-1}(\Delta)) = \card{\psi^{-1}(\{\lfloor t \rfloor\})} \mu(\Delta)
\end{align*}
Consequently, $\hsfn{1}=\hsf_{\phi,1}$ takes the form
\begin{align*}
    \hsfn{1}(x) = \card{\psi^{-1}(\{\lfloor x \rfloor\})}, \quad x \in X.
\end{align*}
Choosing $\psi$ such that $k\mapsto \card{\psi^{-1}(\{k\})}$ is constant on the sets $\psi^{-j}(\{m\})$ for all $m \in \zbb_+$ and $j\in J_n$, but fails to be constant on $\psi^{-(n+1)}(\{m_0\})$ for some $m_0 \in \zbb_+$, we obtain $\phi$ so that the composition operator $C_\phi$ is $n$-centered but not $(n+1)$-centered.
\end{rem}

In the next elementary example we demonstrate ``local'' sensivity of $n$-centeredness. For this we consider a \emph{wco} whose transformation's graph is a rooted binary tree. Such a tree, if valency of a vertex is taken into account, seems to be the simplest ``regular'' generalisation of the rooted directed tree $\zbb_+$ supporting classical unilateral shifts (known to be centered).

\begin{exa}
Let $X = \{0\} \cup \{x_{m,k} \colon k\in J_{2^m}, m \in \nbb\}$, $\ascr = 2^X$, and $\mu$ be the counting measure. We define $\phi \colon X \to X$ by $\phi(0) = 0$, $\phi(x_{1,1}) = \phi(x_{1,2}) = 0$, and
\begin{align*}
    \phi(x_{m,k}) = x_{m-1, \lceil k/2 \rceil}, \quad k\in J_{2^m}, m\in\nbb\setminus{\{1\}}.
\end{align*}
Then $X=\bigsqcup_{m\in\nbb} X_m$ (disjoint union) with $X_m = \{x_{m,k}\colon k\in J_{2^m}\}$. The graph illustrating $\cfw$ is presented in Figure \ref{fig:binary_tree}.

\begin{figure}[ht]
\begin{center}
\begin{tikzpicture}[scale=0.85, transform shape]
\tikzstyle{every node} = [circle,fill=gray!30]

\node (0)[font=\footnotesize, inner sep = 3pt, label={[fill=none]above:$0$}] at (0,0) {$0$};

\node (11)[font=\footnotesize, inner sep = 1pt, label={[fill=none]above:$1$}] at (2,1.25) {$x_{1,1}$};
\node (12)[font=\footnotesize, inner sep = 1pt, label={[fill=none]above:$1$}] at (2,-1.25) {$x_{1,2}$};

\node (21)[font=\footnotesize, inner sep = 1pt, label={[fill=none]above:$1$}] at (4.5,2.25) {$x_{2,1}$};
\node (22)[font=\footnotesize, inner sep = 1pt, label={[fill=none]above:$1$}] at (4.5,0.75) {$x_{2,2}$};
\node (23)[font=\footnotesize, inner sep = 1pt, label={[fill=none]above:$1$}] at (4.5,-0.75) {$x_{2,3}$};
\node (24)[font=\footnotesize, inner sep = 1pt, label={[fill=none]above:$1$}] at (4.5,-2.25) {$x_{2,4}$};

\node (31)[fill=none] at (6, 2.25) {$\cdots$};
\node (32)[fill=none] at (6, 0.75) {$\cdots$};
\node (33)[fill=none] at (6, -0.75) {$\cdots$};
\node (34)[fill=none] at (6, -2.25) {$\cdots$};

\draw[<-] (0) --(11);
\draw[<-] (0) --(12);

\draw[<-] (11) --(21);
\draw[<-] (11) --(22);
\draw[<-] (12) --(23);
\draw[<-] (12) --(24);

\draw[<-, dashed] (21) --(31);
\draw[<-, dashed] (22) --(32);
\draw[<-, dashed] (23) --(33);
\draw[<-, dashed] (24) --(34);

\draw[<-] (0) to[out=135, in=225, looseness=6] (0);
\end{tikzpicture}
\end{center}
\caption{\label{fig:binary_tree} A \emph{wco} induced by a uniform binary structure.}
\end{figure}
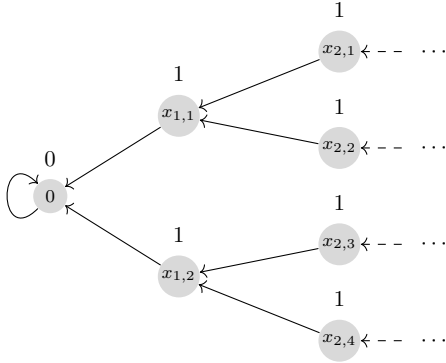

Let $w \colon X \to \rbb_+$ be defined by $w(0) = 0$ and $w(x_{m,k}) = 1$ for all $k\in J_{2^m}$ and $m\in\nbb$. Clearly, $\{w_l \neq 0\} \subseteq \bigcup_{m \Ge l} X_m$. For $x \in X_m$ with $m \Ge l$, we have $\phi^{-l}(\{x\}) = \{y \in X_{m+l} \colon \phi^l(y) = x\}$. Thus, in view of \cite[Prop. 79]{2018-lnim-budzynski-jablonski-jung-stochel},
\begin{align*}
    \hsf_l(x) = \sum_{y \in \phi^{-l}(\{x\})} |w_l(y)|^2 = 2^l \quad \text{a.e. } [\mu_l],\quad l \in \nbb.
\end{align*}
Since $\hsf_l$ is constant on $\{w_l \neq 0\}$, $\efwn{k}(\hsf_l) = \hsf_l$ a.e. $[\mu_k]$ for all $k,l \in \nbb$, making $\cfw$ $n$-centered for every $n\in\nbb$.

Now we locally perturb $\cfw$ by altering a single weight in $X_{n_0+1}$, with a fixed $n_0 \in\nbb$. Namely, we set $w(x_{n_0+1,1}) = a$ with $a > 0$ and $a \neq 1$, while leaving all other weights equal to 1. We have (see \cite[Eq. (6.5)]{2018-lnim-budzynski-jablonski-jung-stochel})
\begin{align*}
    \hsfn{1}(x_{n_0,1}) &= |w(x_{n_0+1,1})|^2 + |w(x_{n_0+1,2})|^2 = a^2 + 1, \\
    \hsfn{1}(x_{n_0,2}) &= |w(x_{n_0+1,3})|^2 + |w(x_{n_0+1,4})|^2 = 1 + 1 = 2.
\end{align*}
Clearly, $\phi^{-1}(\{x_{n_0-1,1}\}) = \{x_{n_0,1}, x_{n_0,2}\}$. Since $a^2 + 1 \neq 2$, $\efwn{1}(\hsfn{1}) = \hsfn{1}$ a.e. $[\mu_1]$ fails to hold, meaning $\cfw$ is not even 1-centered. Consequently, $\cfw$ is not $n$-centered for any other $n\in\nbb$.

Note that it is possible to restore $n$-centeredness by rebalancing other weights. We have two distinct approaches. For $a \Le \sqrt{2}$ we apply the ``sibling fix'', setting  $w(x_{n_0+1,2}) = \sqrt{2-a^2}$. This yields $\hsfn{1}(x_{n_0,1}) = a^2 + 2 - a^2 = 2$. Consequently, $\hsfn{1} = 2$ on $X_{n_0}$, and so $\efwn{n}(\hsfn{1}) = \hsfn{1}$ a.e. $[\mu_{n}]$. Furthermore, since all other weights were unchanged, $\hsf_l = 2^l$ everywhere relative to $\{w_l \neq 0\}$ for all $l \in \nbb$. Thus, $n$-centeredness, with any $n\in\nbb$, is restored. For $a>\sqrt{2}$, we apply the ``cousin fix'', leaving $w(x_{n_0+1,2}) = 1$ and increasing the value of $\hsfn{1}$ across all elements in $X_{n_0}$ that share the same image under $\phi^n$ as $x_{n_0,1}$ to satisfy
\begin{align*}
    |w(x_{n_0+1, 2j-1})|^2 + |w(x_{n_0+1, 2j})|^2 = a^2 + 1, \quad 2 \Le j \Le 2^{\min\{n, n_0\}}.
\end{align*}
Then $\cfw$ is $n$-centered. If $n\geqslant n_0$ and the ``cousin fix'' is applied to the whole of $X_{n_0}$, $\cfw$ becomes automatically $k$-centered for any $k\in\nbb$.
\end{exa}

\section{Related conditions}

Bounded centered {\em wco}'s have recently been characterised in \cite[Th. 6]{2026-arxiv-budzynski-01}. In this section we discuss various finite counterparts of conditions appearing in the above mentioned characterisation and their relevance to spectral $n$-centeredness and weak $n$-centerednes. For the readers convenience we provide a diagram elucidating the relations between the conditions in Figure \ref{fig:relations}.

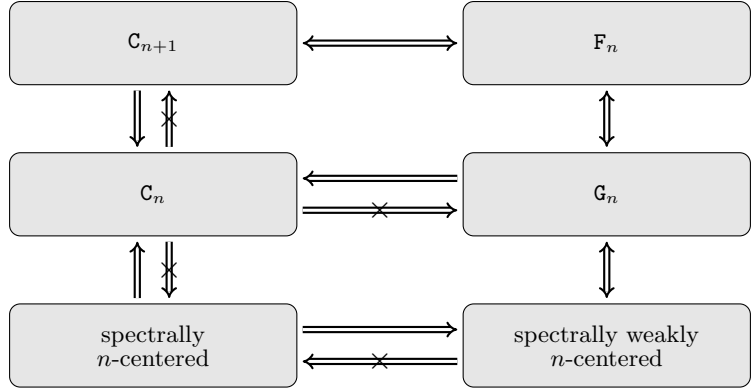
\begin{figure}[ht]
\begin{center}
\begin{tikzpicture}[
    scale=1, transform shape,
    box/.style={draw, rectangle, rounded corners, fill=gray!20, minimum height=1.1cm, minimum width=3.8cm, align=center, font=\small},
    impl/.style={double, double equal sign distance, -{Implies}, thick, shorten >=2pt, shorten <=2pt},
    equiv/.style={double, double equal sign distance, {Implies}-{Implies}, thick, shorten >=2pt, shorten <=2pt}
]

% Nodes arranged exactly as a matrix: a_ij
% Row 1
\node[box] (c1)   at (0, 4) {{\tt C}$_{n+1}$};
\node[box] (fn)   at (6, 4) {{\tt F}$_n$};
% Row 2
\node[box] (cn)   at (0, 2) {{\tt C}$_n$};
\node[box] (gn)   at (6, 2) {{\tt G}$_n$};
% Row 3
\node[box] (snc)  at (0, 0) {spectrally \\[-0.5ex] $n$-centered};
\node[box] (swnc) at (6, 0) {spectrally weakly \\[-0.5ex] $n$-centered};

% Horizontal Arrows
% Row 1: Equivalence between C_{n+1} and F_n
\draw[equiv] (c1) -- (fn);

% Row 2: G_n implies C_n, but C_n does not imply G_n
\draw[impl] ([yshift=6pt]gn.west) -- ([yshift=6pt]cn.east);
\draw[impl] ([yshift=-6pt]cn.east) -- ([yshift=-6pt]gn.west) node[midway] {\Large $\times$};

% Row 3: spectrally n-centered implies weakly, but weakly does not imply spectrally n-centered
\draw[impl] ([yshift=6pt]snc.east) -- ([yshift=6pt]swnc.west);
\draw[impl] ([yshift=-6pt]swnc.west) -- ([yshift=-6pt]snc.east) node[midway] {\Large $\times$};

% Vertical Arrows (Right column - Equivalences)
\draw[equiv] (fn) -- (gn);
\draw[equiv] (gn) -- (swnc);

% Vertical Arrows (Left column - Implications & Non-implications)
% C_{n+1} and C_n
\draw[impl] ([xshift=-6pt]c1.south) -- ([xshift=-6pt]cn.north);
\draw[impl] ([xshift=6pt]cn.north) -- ([xshift=6pt]c1.south) node[midway] {\Large $\times$};

% spectrally n-centered and C_n
\draw[impl] ([xshift=-6pt]snc.north) -- ([xshift=-6pt]cn.south);
\draw[impl] ([xshift=6pt]cn.south) -- ([xshift=6pt]snc.north) node[midway] {\Large $\times$};

\end{tikzpicture}
\end{center}
\caption{\label{fig:relations} Diagram of relations between spectral centeredness properties and associated conditions.}
\end{figure}

The first of the conditions we discuss is the following  {\em cocycle condition}
\begin{enumerate}
\item[{\tt C}$_n$:] for all $k \in J_n$, $\prod_{j=1}^k \hsfn{1}\circ\phi^j = \hsf_k \circ\phi^k$ a.e. $[\mu_k]$.
\end{enumerate}
Since {\tt C}$_1$ is automatically satisfied for all densely defined {\em wco's}, and there are plenty of densely  defined {\em wco}'s that are not spectrally $1$-centered, condition {\tt C}$_n$ itself cannot be sufficient for spectral $n$-centeredness. However, as shown below, it is necessary.

\begin{pro}\label{work01}
Let $n \in \nbb$ and $\cfw \in \cfrak{n}$. If $\cfw$ is spectrally $n$-centered, then {\tt C}$_n$ is satisfied.
\end{pro}
\begin{proof}
We proceed by induction. For $k=1$, the equality in {\tt C}$_n$ holds trivially. Suppose now that the equality in {\tt C}$_n$ holds for $k-1$ where $k \in J_n$ with $k \geqslant 2$. By \cite[Lem. 1]{2026-arxiv-budzynski-01}, \cite[Lem. 5]{2018-lnim-budzynski-jablonski-jung-stochel}, (iii), and $\{p_{\phi,w}^{[k]} \neq 0\} \subseteq \{p_{\phi,w}^{[k-1]} \neq 0\}$, we deduce
\begin{align}\label{bajki01}
    \hsf_k \circ \phi^k = (\hsfn{1} \circ \phi) \cdot (\hsf_{k-1} \circ \phi^k) \quad \text{a.e. } [\mu_k].
\end{align}
By the induction hypothesis and \cite[Lem. 5]{2018-lnim-budzynski-jablonski-jung-stochel}, we get $\hsf_{k-1} \circ \phi^k = \prod_{j=2}^k \hsfn{1} \circ \phi^j$ a.e. $[\mu_k]$. Substituting this back into \eqref{bajki01} yields $\hsf_k \circ \phi^k = \prod_{j=1}^k \hsfn{1} \circ \phi^j$ a.e. $[\mu_k]$, completing the induction step.
\end{proof}

A simple concrete example below shows that the reverse of Proposition \ref{work01} does not hold. 
\begin{exa}
Let $X=\{0, 1, 2, \ldots, 8\}$, $\ascr=2^X$, and $\mu$ be the counting measure. Let $\phi$ and $w$ be defined according to the assignment drawn in Figure \ref{fire01}.
\begin{figure}[ht]
\begin{center}
\begin{tikzpicture}[scale=0.8, transform shape]
\tikzstyle{every node} = [circle,fill=gray!30]

% Vertices with weights as labels above them
\node (0)[font=\footnotesize, inner sep = 3pt, label={[fill=none]above:$0$}] at (0,0) {$0$};

\node (1)[font=\footnotesize, inner sep = 3pt, label={[fill=none]above:$1$}] at (2,1.) {$1$};
\node (3)[font=\footnotesize, inner sep = 3pt, label={[fill=none]above:$1$}] at (4,1.) {$3$};
\node (5)[font=\footnotesize, inner sep = 3pt, label={[fill=none]above:$1$}] at (6,1.) {$5$};
\node (7)[font=\footnotesize, inner sep = 3pt, label={[fill=none]above:$\sqrt{2}$}] at (8,1.) {$7$};

\node (2)[font=\footnotesize, inner sep = 3pt, label={[fill=none]above:$1$}] at (2,-1.) {$2$};
\node (4)[font=\footnotesize, inner sep = 3pt, label={[fill=none]above:$1$}] at (4,-1.) {$4$};
\node (6)[font=\footnotesize, inner sep = 3pt, label={[fill=none]above:$1$}] at (6,-1.) {$6$};
\node (8)[font=\footnotesize, inner sep = 3pt, label={[fill=none]above:$\sqrt{3}$}] at (8,-1.) {$8$};

% Edges without labels
\draw[<-] (0) --(1);
\draw[<-] (1) --(3);
\draw[<-] (3) --(5);
\draw[<-] (5) --(7);

\draw[<-] (0) --(2);
\draw[<-] (2) --(4);
\draw[<-] (4) --(6);
\draw[<-] (6) --(8);

% Self-loop on 0 pointing leftwards
\draw[<-] (0) to[out=135, in=225, looseness=6] (0);

\end{tikzpicture}
\end{center}
\caption{\label{fire01}The graph illustrating $\cfw$ which satisfies {\tt C}$_2$ but is not $2$-centered.}
\end{figure}
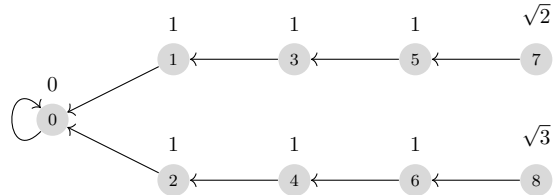
Using \cite[Eq. (6.5)]{2018-lnim-budzynski-jablonski-jung-stochel}, we easily get $\hsfn{1}(7)=\hsfn{1}(8)=0$, $\hsfn{1}(5) = 2=\hsfn{1}(0)$, $\hsfn{1}(6) = 3$, and $\hsfn{1}(x) = 1$ for $x\in \{1, 2, 3, 4\}$. Moreover, $\hsfn{2}(0) = 2$, $\hsfn{2}(1) = 1$, $\hsfn{2}(2) = 1$, $\hsfn{2}(3) = 2$, $\hsfn{2}(4) = 3$, and $\hsfn{2}(x) = 0$ for $x \in \{5, 6, 7, 8\}$. Since $\hsfn{1}$ is constant on sets in $\phi^{-1}(\ascr)$, $\efwn{1}(\hsfn{1}) = \hsfn{1}$ a.e. $[\mu_1]$ holds. Similarly, $\hsfn{1}$ is constant on sets in $\phi^{-2}(\ascr)$ which yields, $\efwn{2}(\hsfn{1}) = \hsfn{1}$ a.e. $[\mu_2]$. Easy calculations show that the cocycle condition $\hsf_2 \circ \phi^2 = (\hsfn{1} \circ \phi) \cdot (\hsfn{1} \circ \phi^2)$ a.e. $[\mu_2]$ is satisfied as well. On the other hand, $\efwn{2}(\hsf_2) = \hsf_2$ fails to be satisfied. Specifically, $\hsfn{2}(3)=2\neq 3=\hsfn{2}(4)$ which means that $\hsfn{2}$ is not constant on $\phi^{-2}(\{0\})$ a.e. $[\mu_2]$. Thus condition (iii) of Theorem \ref{gary02} fails to hold and so $C_\phi$ is not $2$-centered.
\end{exa}

One of the consequences of the Proposition \ref{work01} is that the phase of the $n$-th power $\cfw^n$ of a spectrally $n$-centered $\cfw$ turns out to be the $n$-th power of the phase of $\cfw$, the phenomenon well-known from the bounded case.

\begin{pro}
Let $n \in \nbb$ and $\cfw \in \cfrak{n}$. Let $U = C_{\phi, u_{\phi, w}}$ be the phase of $\cfw$. If {\tt C}$_n$ holds, then the phase of $\cfw^n$ is $U^n$. Moreover, we have
\begin{align*}
    u_{\phi^n, p_{\phi, w}^{[n]}} = p_{\phi, u_{\phi, w}}^{[n]} \quad \text{a.e. } [\mu_n].
\end{align*}
In particular, the claim holds whenever $\cfw$ is spectrally $n$-centered.
\end{pro}
\begin{proof}
Write $w_n = p_{\phi, w}^{[n]}$ and $u=u_{\phi, w}$. Then $U^n = C_{\phi^n, p_{\phi, u}^{[n]}}$ and
\begin{align*}
    p_{\phi, u}^{[n]} = \prod_{j=0}^{n-1} u \circ \phi^j = \prod_{j=0}^{n-1} \frac{w \circ \phi^j}{\sqrt{\hsfn{1} \circ \phi^{j+1}}} = \frac{\prod_{j=0}^{n-1} w \circ \phi^j}{\sqrt{\prod_{j=0}^{n-1} \hsfn{1} \circ \phi^{j+1}}} = \frac{w_n}{\sqrt{\prod_{j=1}^{n} \hsfn{1} \circ \phi^j}} \quad \text{a.e. } [\mu_n].
\end{align*}
On the other hand, since $\cfw \in \cfrak{n}$, $\cfw^n = C_{\phi^n, w_n}$. Hence, weight of its phase is
\begin{align}\label{eq:phase_power_direct}
    u_{\phi^n, w_n} = \frac{w_n}{\sqrt{\hsfn{n} \circ \phi^n}} \quad \text{a.e. } [\mu_n].
\end{align}
Assuming {\tt C}$_n$ holds, we have $\hsfn{n} \circ \phi^n = \prod_{j=1}^n \hsfn{1} \circ \phi^j$ a.e. $[\mu_n]$. Substituting this into \eqref{eq:phase_power_direct}, we deduce $u_{\phi^n, w_n} = p_{\phi, u}^{[n]}$ a.e. $[\mu_n]$. This completes the proof.
\end{proof}

As one can expect, the cocycle condition is related to $\phi^{-k}(\ascr)$-measurability of $\hsfn{1}$.

\begin{pro}\label{bydle}
Let $k\in\nbb$ and $\cfw \in \cfrak{k+1}$. Consider the following conditions:
\begin{itemize}
\item[(i)] for $n \in \{k, k+1\}$, $\hsf_n \circ \phi^n = \prod_{j=1}^n \hsfn{1} \circ \phi^j$ a.e. $[\mu_n]$,
\item[(ii)] $\efwn{k}(\hsfn{1}) = \hsfn{1}$ a.e. $[\mu_k]$.
\end{itemize}
Then (i) implies (ii).
\end{pro}

\begin{proof}
Write $w_m = p_{\phi,w}^{[m]}$. By \cite[Lem 1]{2026-arxiv-budzynski-01} and \cite[Lem. 5]{2018-lnim-budzynski-jablonski-jung-stochel}, we have
\begin{align*}
    \hsf_{k+1} \circ \phi^{k+1} = \efwn{k}(\hsfn{1}) \circ \phi \cdot (\hsf_k \circ \phi^{k+1}) \quad \text{a.e. } [\mu_{k+1}].
\end{align*}
Therefore, using (i), we get
\begin{align*}
    \prod_{j=1}^{k+1} \hsfn{1} \circ \phi^j = \efwn{k}(\hsfn{1}) \circ \phi \cdot \prod_{j=2}^{k+1} \hsfn{1} \circ \phi^j \quad \text{a.e. } [\mu_{k+1}].
\end{align*}
Since $\hsfn{1}\circ\phi^{j+1}>0$ a.e. $[\mu_j]$ (use \cite[Lem. 6 and 26]{2018-lnim-budzynski-jablonski-jung-stochel}), we deduce
\begin{align*}
    \hsfn{1} \circ \phi = \efwn{k}(\hsfn{1}) \circ \phi \quad \text{a.e. } [\mu_{k+1}].
\end{align*}
By \cite[Lem. 5]{2018-lnim-budzynski-jablonski-jung-stochel},  $\hsfn{1} = \chi_{\{\hsfn{1}>0\}}\efwn{k}(\hsfn{1})$ a.e. $[\mu_{k}]$. We now show that $\efwn{k}(\hsfn{1})=0$ a.e. $[\mu_{k}]$ on $\{\hsfn{1}=0\}$. For any $m\in\nbb$, $\big\{\efwn{k}(\hsfn{1}) \Le m\big\}\in\phi^{-k}(\ascr)$. Let $\varDelta\in \phi^{-k}(\ascr)$ satisfy $\mu_k(\varDelta)<\infty$. Set $\varDelta_m= \varDelta\cap \big\{\efwn{k}(\hsfn{1}) \Le m\big\}$. Then
\begin{align*}
    \int_{\varDelta_m} \hsfn{1} \D\mu_k = \int_{\varDelta_m} \efwn{k}(\hsfn{1}) \D\mu_k \Le m \cdot \mu_k(\varDelta_m) < \infty.
\end{align*}
Moreover, we have
\begin{align*}
    \int_{\varDelta_m \cap \{\hsfn{1}>0\}} \efwn{k}(\hsfn{1}) \D\mu_k &+ \int_{\varDelta_m \cap \{\hsfn{1}=0\}} \efwn{k}(\hsfn{1}) \D\mu_k = \int_{\varDelta_m} \efwn{k}(\hsfn{1}) \D\mu_k = \int_{\varDelta_m} \hsfn{1} \D\mu_k \\
    &= \int_{\varDelta_m \cap \{\hsfn{1}>0\}} \hsfn{1} \D\mu_k + \int_{\varDelta_m \cap \{\hsfn{1}=0\}} \hsfn{1} \D\mu_k = \int_{\varDelta_m \cap \{\hsfn{1}>0\}} \hsfn{1} \D\mu_k.
\end{align*}
Since $\efwn{k}(\hsfn{1}) = \hsfn{1}$ a.e. $[\mu_k]$ on $\{\hsfn{1}>0\}$, we conclude that $\efwn{k}(\hsfn{1})= 0$ a.e. $[\mu_k]$ on $\varDelta_m \cap \{\hsfn{1}=0\}$. Using $\sigma$-finiteness of $\mu_k|_{\phi^{-k}(\ascr)}$ we get $\efwn{k}(\hsfn{1}) = 0 = \hsfn{1}$ a.e. $[\mu_k]$ on $\{\hsfn{1}=0\}$. Thus, $\efwn{k}(\hsfn{1}) = \hsfn{1}$ a.e. $[\mu_k]$.
\end{proof}

In view of Proposition \ref{bydle}, {\tt C}$_{n+1}$ yields 
\begin{itemize}
\item[{\tt F}$_n$:] for every $k\in J_n$, $\hsfn{1}=\efwn{k}(\hsfn{1})$ a.e. $[\mu_{k}]$,
\end{itemize}
which is a finite counterpart of another condition appearing in \cite[Th. 6]{2026-arxiv-budzynski-01}. Clearly, by Lemma \ref{gary01}, condition {\tt F}$_n$ is equivalent to
\begin{itemize}
    \item for every $k\in J_n$, $\pfwn{k} \mno{1}\subseteq \mno{1}\pfwn{k}$.
\end{itemize}
Moreover, by inspecting the proof of Theorem \ref{gary02}, one can deduce the following.
\begin{pro}\label{wnc}
Let $n\in\nbb$ and $\cfw\in \cfrak{n}$. Then $\cfw$ is spectrally weakly $n$-centered if and only if {\tt F}$_n$ holds.
\end{pro}

The relation between the cocycle condition {\tt C}$_n$ and the $\phi^{-k}(\ascr)$-measurability of $\hsfn{1}$ is a delicate one. Below, for a fixed $m \in \nbb$, we construct $\cfw$ such that the cocycle condition holds for $k \in J_{m}$, the measurability condition $\efwn{k}(\hsfn{1}) = \hsfn{1}$ a.e. $[\mu_k]$ holds for $k\in J_{m-1}$, but $\efwn{m}(\hsfn{1}) = \hsfn{1}$ a.e. $[\mu_m]$ fails.
\begin{exa}
Let $m\in\nbb$. Let $X = \{0\} \cup \big\{(i,j) \colon i \in \{1,2\}, j \in \nbb\big\}$, let $\ascr = 2^X$, and let $\mu$ be the counting measure. Let $\phi \colon X \to X$ and $w \colon X \to \rbb_+$ be defined according to the assignment drawn in Figure \ref{fig:measurability_gap_general}.

%Specifically, we have $\phi(0) = 0$, $\phi((i,1)) = 0$, and $\phi((i,j)) = (i, j-1)$ for $j \geqslant 2$. The weight function is defined as $w(0)=0$, $w((1,1)) = w((2,1)) = \frac{1}{\sqrt{2}}$, $w((1,m+1)) = \sqrt{2}$, $w((2,m+1)) = \sqrt{3}$, and $w(x) = 1$ for all other $x \in X$.

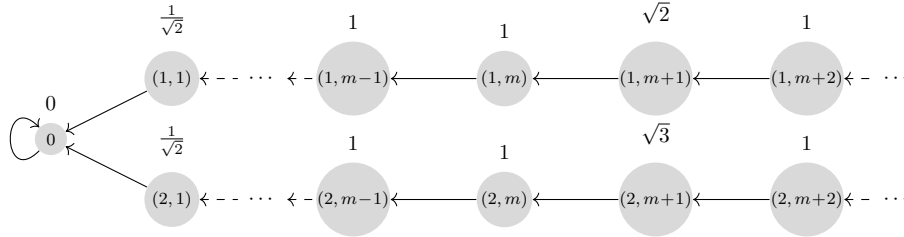
\begin{figure}[ht]
\begin{center}
\begin{tikzpicture}[scale=0.8, transform shape]
\tikzstyle{every node} = [circle,fill=gray!30]

\node (0)[font=\footnotesize, inner sep = 3pt, label={[fill=none]above:$0$}] at (0,0) {$0$};

\node (11)[font=\footnotesize, inner sep = 2pt, label={[fill=none]above:$\frac{1}{\sqrt{2}}$}] at (2,1.) {$(1,1)$};
\node (1dots)[fill=none] at (3.5,1.) {$\cdots$};
\node (1m1)[font=\footnotesize, inner sep = 0pt, label={[fill=none]above:$1$}] at (5,1.) {$(1,m\!-\!1)$};
\node (1m)[font=\footnotesize, inner sep = 1pt, label={[fill=none]above:$1$}] at (7.5,1.) {$(1,m)$};
\node (1m1p)[font=\footnotesize, inner sep = 0pt, label={[fill=none]above:$\sqrt{2}$}] at (10,1.) {$(1,m\!+\!1)$};
\node (1m2p)[font=\footnotesize, inner sep = 0pt, label={[fill=none]above:$1$}] at (12.5,1.) {$(1,m\!+\!2)$};
\node (1inf)[fill=none] at (14,1.) {$\cdots$};

\node (21)[font=\footnotesize, inner sep = 2pt, label={[fill=none]above:$\frac{1}{\sqrt{2}}$}] at (2,-1.) {$(2,1)$};
\node (2dots)[fill=none] at (3.5,-1.) {$\cdots$};
\node (2m1)[font=\footnotesize, inner sep = 0pt, label={[fill=none]above:$1$}] at (5,-1.) {$(2,m\!-\!1)$};
\node (2m)[font=\footnotesize, inner sep = 1pt, label={[fill=none]above:$1$}] at (7.5,-1.) {$(2,m)$};
\node (2m1p)[font=\footnotesize, inner sep = 0pt, label={[fill=none]above:$\sqrt{3}$}] at (10,-1.) {$(2,m\!+\!1)$};
\node (2m2p)[font=\footnotesize, inner sep = 0pt, label={[fill=none]above:$1$}] at (12.5,-1.) {$(2,m\!+\!2)$};
\node (2inf)[fill=none] at (14,-1.) {$\cdots$};

\draw[<-] (0) --(11);
\draw[<-, dashed] (11) --(1dots);
\draw[<-, dashed] (1dots) --(1m1);
\draw[<-] (1m1) --(1m);
\draw[<-] (1m) --(1m1p);
\draw[<-] (1m1p) --(1m2p);
\draw[<-, dashed] (1m2p) --(1inf);

\draw[<-] (0) --(21);
\draw[<-, dashed] (21) --(2dots);
\draw[<-, dashed] (2dots) --(2m1);
\draw[<-] (2m1) --(2m);
\draw[<-] (2m) --(2m1p);
\draw[<-] (2m1p) --(2m2p);
\draw[<-, dashed] (2m2p) --(2inf);

\draw[<-] (0) to[out=135, in=225, looseness=6] (0);
\end{tikzpicture}
\end{center}
\caption{\label{fig:measurability_gap_general} The graph illustrating $\cfw$ which satisfies {\tt C}$_m$, but $\hsfn{1}$ fails to be $\phi^{-m}(\ascr)$-measurabile.}
\end{figure}

Using \cite[Eq. (6.5)]{2018-lnim-budzynski-jablonski-jung-stochel}, one shows that 
\begin{align*}
    \hsfn{1}(x) &= 1,\quad x\in X\setminus \{ (1,m), (2,m)\},
\end{align*}
and
\begin{align*}
    \hsfn{1}((1,m)) &= |w((1,m+1))|^2 = 2, \\
    \hsfn{1}((2,m)) &= |w((2,m+1))|^2 = 3.
\end{align*}
We also have
\begin{align*}
\text{$\phi^k((1,m)) = (1, m-k)$ and $\phi^k((2,m)) = (2, m-k)$}\quad  k\in J_{m-1}
\end{align*}
Since $m-k \geqslant 1$, $\phi^{-k}(\{(1,m-k)\})\cap \phi^{-k}(\{(2,m-k)\})=\varnothing$. Thus $\hsfn{1}$ is constant (up to $\mu_k$) on sets $\phi^{-k}(x)$. Consequently, $\efwn{k}(\hsfn{1}) = \hsfn{1}$ a.e. $[\mu_k]$ for $k \in J_{m-1}$. 

We now show that $\efwn{m}(\hsfn{1}) = \hsfn{1}$ a.e. $[\mu_m]$ fails. Both $(1,m)$ and $(2,m)$ belong to $\phi^{-m}(\{0\})$. Furthermore, $w_m((i,m)) = \prod_{j=0}^{m-1} w((i, m-j)) = 1 \cdots 1 \cdot \frac{1}{\sqrt{2}} \neq 0$ for $i=1, 2$. However, $\hsfn{1}((1,m)) = 2 \neq 3 = \hsfn{1}((2,m))$. Hence,  $\hsfn{1}$ is not constant on $\phi^{-m}(\{0\})$ and so $\efwn{m}(\hsfn{1}) = \hsfn{1}$ a.e. $[\mu_m]$ is not satisfied.

Finally, we show that the cocycle condition {\tt C}$_m$ is satisfied. Set $w_k = p_{\phi,w}^{[k]}$. Fix $k\in J_m$ and $x \in X$ such that $w_k(x) \neq 0$. Then $x = (i, j)$ for $j \geqslant k$ (otherwise $\phi^{k-1}(x)=0$ and thus $w_k(x)$ vanishes). If $j=k$, we have $\phi^k(x) = 0$ and the right-hand side of the equality in {\tt C}$_m$ is 
\begin{align*}
\hsfn{k} \circ \phi^k(x)=\hsfn{k}(0) = \sum_{y \in \phi^{-k}(\{0\})} |w_k(y)|^2
\end{align*}
Since $\phi^{-k}(\{0\})=\{0\}\cup \{(i,j)\colon i=1,2, \ j\in J_{k}\}$ and $w_k((i,j))=0$ for $i=1,2$ and $j\in J_{k-1}$, we see that 
\begin{align*}
\hsfn{k} \circ \phi^k(x)=|w_k((1,k))|^2 + |w_k((2,k))|^2 = \frac{1}{2} + \frac{1}{2} = 1.
\end{align*}
The left-hand side is 
\begin{align*}
\prod_{l=1}^k \hsfn{1} \circ \phi^l(x)=\prod_{l=1}^k \hsfn{1}((i,k-l))=1.
\end{align*}
Combining the two above we settle the case $j=k$. If $j>k$, $\phi^k(x) = (i, j-k)$. Thus the right-hand side of the equality in {\tt C}$_m$ is
\begin{align*}
    \hsfn{k} \circ \phi^k(x) = \hsfn{k}((i, j-k)) = |w_k((i,j))|^2 
    = \prod_{l=0}^{k-1} |w(\phi^l((i,j)))|^2 = \prod_{l=0}^{k-1} |w((i, j-l))|^2.
\end{align*}
The left-hand side is 
\begin{align*}
    \prod_{l=1}^k \hsfn{1} \circ \phi^l(x) = \prod_{l=1}^k \hsfn{1}((i, j-l)) = \prod_{l=1}^k |w((i, j-l+1))|^2 
    = \prod_{r=0}^{k-1} |w((i, j-r))|^2.
\end{align*}
This settles the case $j>k$ and completes the proof that {\tt C}$_m$ is satisfied.
\end{exa}
On the positive note we have the following.
\begin{pro}\label{fn_implies_cn1}
Let $n \in \nbb$ and $\cfw \in \cfrak{n+1}$. Then {\tt F}$_n$ is equivalent to {\tt C}$_{n+1}$.
\end{pro}

\begin{proof}
That {\tt C}$_{n+1}$ implies {\tt F}$_n$ follows from Proposition \ref{bydle}. For the proof of the reverse implication, we proceed by induction on $k$ in {\tt C}$_{n+1}$. The case $k=1$ holds trivially. Assume now that the equality holds for some $k \in J_n$, meaning
\begin{align}\label{eq:ind_hyp_fn}
    \hsf_k \circ \phi^k = \prod_{j=1}^k \hsfn{1} \circ \phi^j \quad \text{a.e. } [\mu_k].
\end{align}
By \cite[Lem. 1]{2026-arxiv-budzynski-01} and \cite[Lem. 5]{2018-lnim-budzynski-jablonski-jung-stochel}, we have the identity
\begin{align}\label{eq:fund_ident_fn}
    \hsf_{k+1} \circ \phi^{k+1} = \efwn{k}(\hsfn{1}) \circ \phi \cdot (\hsf_k \circ \phi^{k+1}) \quad \text{a.e. } [\mu_{k+1}].
\end{align}
Since $\mu_{k+1}$ is absolutely continuous with respect to $\mu_k \circ \phi^{-1}$ (see \cite[Lem. 26]{2018-lnim-budzynski-jablonski-jung-stochel}), using {\tt F}$_n$ we get
\begin{align*}
    \efwn{k}(\hsfn{1}) \circ \phi = \hsfn{1} \circ \phi \quad \text{a.e. } [\mu_{k+1}].
\end{align*}
Substituting this into \eqref{eq:fund_ident_fn} gives
\begin{align}\label{eq:fund_ident_sub_fn}
    \hsf_{k+1} \circ \phi^{k+1} = (\hsfn{1} \circ \phi) \cdot (\hsf_k \circ \phi^{k+1}) \quad \text{a.e. } [\mu_{k+1}].
\end{align}
Using the induction hypothesis \eqref{eq:ind_hyp_fn} and the absolute continuity of $\mu_{k+1}$ with respect to $\mu_k \circ \phi^{-1}$ once again, we can shift the functions to obtain
\begin{align*}
    \hsf_k \circ \phi^{k+1} = \prod_{j=1}^k \hsfn{1} \circ \phi^{j+1} \quad \text{a.e. } [\mu_{k+1}].
\end{align*}
Inserting this into \eqref{eq:fund_ident_sub_fn}, we conclude the induction step, which completes the proof.
\end{proof}

Deducing $\phi^{-k}(\ascr)$-measurability of $\hsfn{l}$ with $l$ other than $1$ from {\tt F}$_n$ is possible under some restrictions on $k$ and $l$, as shown in Proposition \ref{kupa01} below. However, it is possible to get $\efwn{k}(\hsf_l) = \hsf_l$ a.e. $[\mu_k]$ with $k, l$ not limited by any constraints whenever passing $n$ to $+\infty$ is possible, i.e. all {\tt C}$_n$ or {\tt F}$_n$, $n\in\nbb$, are satisfied.

\begin{pro}\label{kupa01}
Let $n\in\nbb$ and $\cfw \in \cfrak{n}$. Assume that {\tt F}$_n$ hold. Then $\efwn{k}(\hsf_l) = \hsf_l$ a.e. $[\mu_k]$ for all $k, l \in J_n$ such that $k+l \leqslant n+1$.
\end{pro}

\begin{proof}
The case $l=1$ corresponds to $\efwn{k}(\hsfn{1}) = \hsfn{1}$ a.e. $[\mu_k]$, which holds trivially for all $k \in J_n$ by {\tt F}$_n$.

Suppose the claim holds for $l \in J_{n-1}$, meaning $\efwn{k}(\hsf_l) = \hsf_l$ a.e. $[\mu_k]$ for all $k$ such that $k+l \leqslant n+1$. Fix $k$ such that $k+(l+1) \leqslant n+1$ (thus $k+l \leqslant n$). By \cite[Lem. 26]{2018-lnim-budzynski-jablonski-jung-stochel} and the induction hypothesis, we have
\begin{align}\label{eq:iv_iii_lucek}
    |w_k|^2 \hsf_{l+1} = |w_k|^2 \big(\esf_l(\hsfn{1}) \circ \phi^{-l}\big) \cdot \hsf_l =|w_k|^2 \big(\esf_l(\hsfn{1}) \circ \phi^{-l}\big) \cdot\efwn{k}(\hsf_l)\quad \text{a.e. } [\mu].
\end{align}
Since, by {\tt F}$_n$, $\efwn{l}(\hsfn{1}) = \hsfn{1}$ a.e. $[\mu_l]$, we have 
\begin{align}\label{work02}
    |w_l|^2 \hsfn{1} = |w_l|^2 (H_{l+1} \circ \phi^l) \quad \text{a.e. } [\mu],
\end{align}
with $H_{l+1} = \esf_l(\hsfn{1}) \circ \phi^{-l}$. Furthermore, since $k+l \leqslant n$, {\tt F}$_n$ guarantees that $\efwn{k+l}(\hsfn{1}) = \hsfn{1}$ a.e. $[\mu_{k+l}]$. Thus there exists an $\ascr$-measurable function $G_{k+l}$ such that
\begin{align}\label{work03}
    |w_{k+l}|^2 \hsfn{1} = |w_{k+l}|^2 (G_{k+l} \circ \phi^{k+l}) \quad \text{a.e. } [\mu].
\end{align}
Combining \eqref{work02} and \eqref{work03}, and using $w_{k+l} = w_l \cdot (w_k \circ \phi^l)$, we get
\begin{align*}
    |w_l|^2 \big(|w_k|^2 H_{l+1}\big) \circ \phi^l = |w_l|^2 \big(|w_k|^2 (G_{k+l} \circ \phi^k)\big) \circ \phi^l \quad \text{a.e. } [\mu].
\end{align*}
Applying \cite[Lemma 5]{2018-lnim-budzynski-jablonski-jung-stochel}, we obtain
\begin{align}\label{work04}
    \chi_{\{\hsf_l > 0\}} \cdot |w_k|^2 H_{l+1} = \chi_{\{\hsf_l > 0\}} \cdot |w_k|^2 (G_{k+l} \circ \phi^k) \quad \text{a.e. } [\mu].
\end{align}
By the induction hypothesis, $\efwn{k}(\hsf_l) = \hsf_l$ a.e. $[\mu_k]$. Therefore, $\{\hsf_l > 0\} \cap \{w_k \neq 0\} \in \phi^{-k}(\ascr) \cap \{w_k \neq 0\}$, which yields $|w_k|^2 \chi_{\{\hsf_l > 0\}} = |w_k|^2 (K \circ \phi^k)$ a.e. $[\mu]$ for some $\ascr$-measurable function $K$. Substituting this into \eqref{work04} gives
\begin{align*}
    \chi_{\{\hsf_l > 0\}} \cdot |w_k|^2 H_{l+1} = |w_k|^2 (K \circ \phi^k) \cdot (G_{k+l} \circ \phi^k) \quad \text{a.e. } [\mu].
\end{align*}
Returning to \eqref{eq:iv_iii_lucek}, and noting that $\hsf_l = \chi_{\{\hsf_l > 0\}} \hsf_l$, we can substitute the above identity to find
\begin{align*}
    |w_k|^2 \hsf_{l+1} = |w_k|^2 (K \circ \phi^k) \cdot (G_{k+l} \circ \phi^k) \cdot \efwn{k}(\hsf_l) \quad \text{a.e. } [\mu].
\end{align*}
%Since $\efwn{k}(\hsf_l)$ is $\phi^{-k}(\ascr)$-measurable on $\{w_k \neq 0\}$, all terms on the right-hand side are $\phi^{-k}(\ascr)$-measurable on $\{w_k \neq 0\}$. 
Thus, $\hsf_{l+1}$ is $\phi^{-k}(\ascr)$-measurable relative to $\{w_k \neq 0\}$. Consequently, $\efwn{k}(\hsf_{l+1}) = \hsf_{l+1}$ a.e. $[\mu_k]$, which, by induction, completes the proof.
\end{proof}

Another condition inspired by the bounded-case characterisation of centeredness is
\begin{itemize}
\item[{\tt G}$_n$:] for every $k\in J_n$, $\efwn{k}\big(\hsfn{1}\big)\circ \phi =\efwn{n+1}\big(\hsfn{1}\circ\phi \big)$ a.e. $[\mu_{n+1}]$.
\end{itemize}
This condition turns out to be equivalent to spectral weak $n$-centeredness (cf. Proposition \ref{wnc}).

\begin{pro}
Let $n \in \nbb$ and $\cfw \in \cfrak{n}$. Consider the following conditions:
\begin{itemize}
\item[(i)] $\efwn{n}(\hsfn{1}) = \hsfn{1}$ a.e. $[\mu_n]$,
\item[(ii)] $\efwn{n}(\hsfn{1}) \circ \phi = \efwn{n+1}(\hsfn{1} \circ \phi)$ a.e. $[\mu_{n+1}]$.
\end{itemize}
Then (i) and (ii) are equivalent.
\end{pro}

\begin{proof}
(i) $\Rightarrow$ (ii) The equality $\efwn{n}(\hsfn{1}) = \hsfn{1}$ a.e. $[\mu_n]$ implies 
\begin{align}\label{eq:shifted_F}
    \efwn{n}(\hsfn{1}) \circ \phi = \hsfn{1} \circ \phi \quad \text{a.e. } [\mu_{n+1}],
\end{align}
meaning that $\hsfn{1} \circ \phi$ is $\phi^{-(n+1)}(\ascr)$-measurable relative to $\{w_n \circ \phi \neq 0\}$. Because $\{w_{n+1} \neq 0\} \subseteq \{w_n \circ \phi \neq 0\}$, the function $\hsfn{1} \circ \phi$ is $\phi^{-(n+1)}(\ascr)$-measurable relative to $\{w_{n+1} \neq 0\}$. Therefore,
\begin{align*}
    \efwn{n+1}(\hsfn{1} \circ \phi) = \hsfn{1} \circ \phi \quad \text{a.e. } [\mu_{n+1}].
\end{align*}
Combining this with \eqref{eq:shifted_F} proves (ii).

(ii) $\Rightarrow$ (i) Because $\cfw \in \cfrak{n}$, the measure $\mu_n$ is $\sigma$-finite on $\phi^{-n}(\ascr)$. Let $\varDelta \in \ascr$ be such that $\mu_n(\phi^{-n}(\varDelta)) < \infty$. Define $\varOmega = \phi^{-n}(\varDelta) \in \phi^{-n}(\ascr)$. Since $\efwn{n}(\hsfn{1})$ is $\phi^{-n}(\ascr)$-measurable, for any $m \in \nbb$, the set $\varOmega_m := \varOmega \cap \{\efwn{n}(\hsfn{1}) \Le m\}$ belongs to $\phi^{-n}(\ascr)$. Put $A_m = \phi^{-1}(\varOmega_m) \in \phi^{-(n+1)}(\ascr)$. Then
\begin{align*}
    \int_{A_m} (\efwn{n}(\hsfn{1}) \circ \phi) \D\mu_{n+1} = \int_{A_m} \efwn{n+1}(\hsfn{1} \circ \phi) \D\mu_{n+1} = \int_{A_m} (\hsfn{1} \circ \phi) \D\mu_{n+1}.
\end{align*}
Using $\D\mu_{n+1} = |w_1|^2 |w_n \circ \phi|^2 \D\mu$, we get
\begin{align*}
    \int_{A_m} (\efwn{n}(\hsfn{1}) \circ \phi) \D\mu_{n+1} &= \int_{\phi^{-1}(\varOmega_m)} (\efwn{n}(\hsfn{1}) |w_n|^2) \circ \phi \, |w_1|^2 \D\mu = \int_{\varOmega_m} \efwn{n}(\hsfn{1}) \hsfn{1} \D\mu_n.
\end{align*}
Applying the same again yields
\begin{align*}
    \int_{A_m} (\hsfn{1} \circ \phi) \D\mu_{n+1} &= \int_{\phi^{-1}(\varOmega_m)} (\hsfn{1} |w_n|^2) \circ \phi \, |w_1|^2 \D\mu = \int_{\varOmega_m} \hsfn{1}^2 \D\mu_n.
\end{align*}
Comparing the two, we obtain
\begin{align*}
    \int_{C_m} \efwn{n}(\hsfn{1}) \hsfn{1} \D\mu_n = \int_{C_m} \hsfn{1}^2 \D\mu_n.
\end{align*}
From $\efwn{n}(\hsfn{1}) \Le m$ on $\varOmega_m$, we get
\begin{align*}
    \int_{\varOmega_m} \efwn{n}(\hsfn{1}) \hsfn{1} \D\mu_n \Le m \int_{\varOmega_m} \hsfn{1} \D\mu_n = m \int_{\varOmega_m} \efwn{n}(\hsfn{1}) \D\mu_n \Le m^2 \cdot \mu_n(\varOmega_m) < \infty.
\end{align*}
Therefore, $\int_{\varOmega_m} \hsfn{1}^2 \D\mu_n < \infty$. Since $\efwn{n}$ is an orthogonal projection and $\chi_{\varOmega_m} \efwn{n}(\hsfn{1})$ is $\phi^{-n}(\ascr)$-measurable, we have $\int_{\varOmega_m} \efwn{n}(\hsfn{1})^2 \D\mu_n = \int_{\varOmega_m} \efwn{n}(\hsfn{1}) \hsfn{1} \D\mu_n$. Hence\allowdisplaybreaks
\begin{align*}
    \int_{C_m} (\hsfn{1} - \efwn{n}(\hsfn{1}))^2 \D\mu_n &= \int_{C_m} \hsfn{1}^2 \D\mu_n - 2 \int_{C_m} \efwn{n}(\hsfn{1}) \hsfn{1} \D\mu_n + \int_{C_m} \efwn{n}(\hsfn{1})^2 \D\mu_n \\
    &= \int_{C_m} \hsfn{1}^2 \D\mu_n - 2 \int_{C_m} \hsfn{1}^2 \D\mu_n + \int_{C_m} \hsfn{1}^2 \D\mu_n = 0.
\end{align*}
This implies $\efwn{n}(\hsfn{1}) = \hsfn{1}$ a.e. $[\mu_n]$ on $\varOmega_m$. Applying standard measure-theoretic techniques yields $\efwn{n}(\hsfn{1}) = \hsfn{1}$ a.e. $[\mu_n]$.
\end{proof}

We will close the section with two elementary examples. The first one proves that {\tt C}$_n$ does not imply {\tt G}$_n$. The second one shows that spectral weak $n$-centeredness does not yield spectral $n$-centeredness.

\begin{exa}
Let $X = \{0\} \cup \{(i,j) \colon i \in \{1,2\}, j \in \{1,2,3\}\}$, let $\ascr = 2^X$, and let $\mu$ be the counting measure on $X$. Let $\phi \colon X \to X$ be given by $\phi(0) = 0$, $\phi(1,1) = \phi(2,1) = 0$, and 
\begin{align*}
    \phi(i, j) &= (i, j-1), \quad i \in \{1,2\}, \ j \in \{2,3\},
\end{align*}
and $w \colon X \to \rbb_+$ be given by $w(0) = 0$, $w(1,1) = w(2,1) = 1/\sqrt{2}$, $w(1,2) = w(2,2) = 1$, $w(1,3) = \sqrt{2}$, and $w(2,3) = \sqrt{3}$. Let $\cfw$ be the weighted composition operator induced by $\phi$ and $w$. The graph illustrating $\cfw$ is presented in Figure \ref{fig:cocycle_not_expectation}. Clearly, $\cfw$ is bounded. 

\begin{figure}[ht]
\begin{center}
\begin{tikzpicture}[scale=0.9, transform shape]
\tikzstyle{every node} = [circle,fill=gray!30]

\node (0)[font=\footnotesize, inner sep = 3pt, label={[fill=none]above:$0$}] at (0,0) {$0$};

\node (11)[font=\footnotesize, inner sep = 1pt, label={[fill=none]above:$\frac{1}{\sqrt{2}}$}] at (2,1.) {$(1,1)$};
\node (12)[font=\footnotesize, inner sep = 1pt, label={[fill=none]above:$1$}] at (4,1.) {$(1,2)$};
\node (13)[font=\footnotesize, inner sep = 1pt, label={[fill=none]above:$\sqrt{2}$}] at (6,1.) {$(1,3)$};

\node (21)[font=\footnotesize, inner sep = 1pt, label={[fill=none]above:$\frac{1}{\sqrt{2}}$}] at (2,-1.) {$(2,1)$};
\node (22)[font=\footnotesize, inner sep = 1pt, label={[fill=none]above:$1$}] at (4,-1.) {$(2,2)$};
\node (23)[font=\footnotesize, inner sep = 1pt, label={[fill=none]above:$\sqrt{3}$}] at (6,-1.) {$(2,3)$};

\draw[<-] (0) --(11);
\draw[<-] (11) --(12);
\draw[<-] (12) --(13);

\draw[<-] (0) --(21);
\draw[<-] (21) --(22);
\draw[<-] (22) --(23);

\draw[<-] (0) to[out=135, in=225, looseness=6] (0);
\end{tikzpicture}
\end{center}
\caption{\label{fig:cocycle_not_expectation} The graph related to a \emph{wco} satisfying {\tt C}$_2$ but failing {\tt G}$_2$.}
\end{figure}
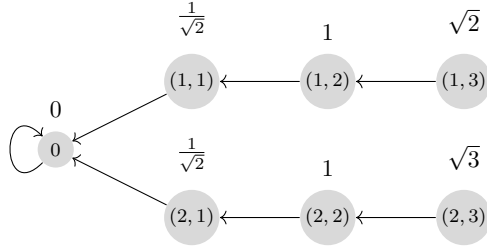

Using $\hsfn{1}(x) = \sum_{y \in \phi^{-1}(\{x\})} |w(y)|^2$ and $\hsf_2(x) = \sum_{y \in \phi^{-2}(\{x\})} |w_2(y)|^2$, where $w_2=w\cdot w\circ \phi$, we obtain
\begin{align}\label{ye02}
    \hsfn{1}(0) = \hsfn{1}(1,1) = \hsfn{1}(2,1) = 1,\quad \hsfn{1}(1,2) = 2,\quad \hsfn{1}(2,2) = 3,
\end{align}
and
\begin{align}\label{ye03}
    \hsf_2(0) = 1,\quad \hsf_2(1,1) = 2,\quad \hsf_2(2,1) = 3.
\end{align}
Then $\{ w_2\neq 0\}=\{(1,2), (2,2), (1,3), (2,3)\}$. Using this, \eqref{ye02}, and \eqref{ye03}, we easily verify that $\hsf_2 \circ \phi^2 = (\hsfn{1} \circ \phi) \cdot (\hsfn{1} \circ \phi^2)$ holds on $\{w_2\neq 0\}$, which means that ({\tt C}$_2$) is satisfied.

On the other hand $\{w_3\neq0\}=w\cdot w\circ \phi \cdot w\circ\phi^2$ is equal to $\{(1,3), (2,3)\}$. Moreover, we have\allowdisplaybreaks
\begin{align*}
    \efwn{2}(\hsfn{1})(1,2) &= \frac{\hsfn{1}(1,2)\mu_2(\{(1,2)\}) + \hsfn{1}(2,2)\mu_2(\{(2,2)\})}{\mu_2(\{(1,2)\}) + \mu_2(\{(2,2)\})} = \frac{5}{2},
\end{align*}
and
\begin{align*}
    \efwn{3}(\hsfn{1} \circ \phi)(1,3) &= \frac{\hsfn{1}(1,2)\mu_3(\{(1,3)\}) + \hsfn{1}(2,2)\mu_3(\{(2,3)\})}{\mu_3(\{(1,3)\}) + \mu_3(\{(2,3)\})} = \frac{13}{5}.
\end{align*}
Consequently, $\efwn{2}(\hsfn{1}) \circ \phi = \efwn{3}(\hsfn{1} \circ \phi)$ a.e. $[\mu_3]$ does not hold, i.e., ({\tt G}$_2$) is not satisfied.
\end{exa}

%\begin{exa}
%Let $X = \{0, 1, 2, 3, 4\}$ equipped with the counting measure $\mu$. Let $\phi \colon X \to X$ be defined by $\phi(0)=0$, $\phi(1)=\phi(2)=0$, $\phi(3)=1$, and $\phi(4)=2$. Let the weight $w \colon X \to \rbb_+$ be given by $w(0)=0$, $w(1)=w(2)=w(3)=1$, and $w(4)=\sqrt{2}$. 

%Since $n=1$, the cocycle condition ${\tt C}_1$ reduces to the trivial identity $\hsfn{1} \circ \phi = \hsfn{1} \circ \phi$, which is automatically satisfied. Thus, $\cfw$ satisfies ${\tt C}_1$.

%We now check condition ${\tt F}_1$, which requires $\efwn{1}(\hsfn{1}) = \hsfn{1}$ a.e. $[\mu_1]$. The support of $w_1 = w$ is $\{1, 2, 3, 4\}$. We evaluate the Radon--Nikodym derivative $\hsfn{1}$ at the siblings $1$ and $2$:
%\begin{align*}
%    \hsfn{1}(1) &= \sum_{y \in \phi^{-1}(\{1\})} |w(y)|^2 = |w(3)|^2 = 1, \\
%    \hsfn{1}(2) &= \sum_{y \in \phi^{-1}(\{2\})} |w(y)|^2 = |w(4)|^2 = 2.
%\end{align*}
%Both $1$ and $2$ belong to the fiber $\phi^{-1}(\{0\})$ and both have non-zero weight. Since $\hsfn{1}(1) \neq \hsfn{1}(2)$, the function $\hsfn{1}$ is not constant on the sets of the form $\phi^{-1}(\{x\})$. Consequently, $\efwn{1}(\hsfn{1}) = \hsfn{1}$ a.e. $[\mu_1]$ fails. 

%Therefore, $\cfw$ satisfies ${\tt C}_1$, but fails ${\tt F}_1$ (and by equivalence, fails ${\tt G}_1$ and is not spectrally weakly 1-centered).
%\end{exa}

\begin{exa}
Let $X = \{0, 1, 2, 3, 4, 5, 6, 7, 8, 9\}$, let $\ascr = 2^X$, and let $\mu$ be the counting measure. We define $\phi \colon X \to X$ and $w\colon X\to \cbb$ according to the assignment drawn in Figure \ref{fig:weak_not_full}. 

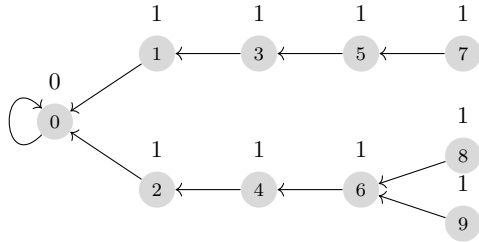
\begin{figure}[ht]
\begin{center}
\begin{tikzpicture}[scale=0.9, transform shape]
\tikzstyle{every node} = [circle,fill=gray!30]

\node (0)[font=\footnotesize, inner sep = 3pt, label={[fill=none]above:$0$}] at (0,0) {$0$};

\node (1)[font=\footnotesize, inner sep = 3pt, label={[fill=none]above:$1$}] at (1.5, 1) {$1$};
\node (2)[font=\footnotesize, inner sep = 3pt, label={[fill=none]above:$1$}] at (1.5,-1) {$2$};

\node (3)[font=\footnotesize, inner sep = 3pt, label={[fill=none]above:$1$}] at (3, 1) {$3$};
\node (4)[font=\footnotesize, inner sep = 3pt, label={[fill=none]above:$1$}] at (3,-1) {$4$};

\node (5)[font=\footnotesize, inner sep = 3pt, label={[fill=none]above:$1$}] at (4.5, 1) {$5$};
\node (6)[font=\footnotesize, inner sep = 3pt, label={[fill=none]above:$1$}] at (4.5,-1) {$6$};

\node (7)[font=\footnotesize, inner sep = 3pt, label={[fill=none]above:$1$}] at (6, 1) {$7$};
\node (8)[font=\footnotesize, inner sep = 3pt, label={[fill=none]above:$1$}] at (6,-0.5) {$8$};
\node (9)[font=\footnotesize, inner sep = 3pt, label={[fill=none]above:$1$}] at (6,-1.5) {$9$};

\draw[<-] (0) --(1);
\draw[<-] (0) --(2);
\draw[<-] (1) --(3);
\draw[<-] (2) --(4);
\draw[<-] (3) --(5);
\draw[<-] (4) --(6);
\draw[<-] (5) --(7);
\draw[<-] (6) --(8);
\draw[<-] (6) --(9);

% Self-loop pointing leftwards
\draw[<-] (0) to[out=135, in=225, looseness=6] (0);
\end{tikzpicture}
\end{center}
\caption{\label{fig:weak_not_full} The graph related to a \emph{wco} that is weakly 2-centered but not 2-centered.}
\end{figure}

Since $\hsfn{1}(x) = \card{\phi^{-1}(\{x\})}$, we $\hsfn{1}(6) = 2$, $\hsfn{1}(0) = 2$, $\hsfn{1}(7)=\hsfn{1}(8)=\hsfn{1}(9)=0$, and $\hsfn{1}(x) = 1$ for all other $x$.

We first check that $\cfw$ is spectrally weakly 2-centered. By Proposition \ref{wnc}, this is equivalent to condition {\tt F}$_2$, which requires $\hsfn{1}$ to be constant on sets of the form $\phi^{-1}(\{x\})$ and $\phi^{-2}(\{x\})$. Easily, this is satisfied. Thus, $\cfw$ is spectrally weakly 2-centered.

For spectral 2-centeredness $\efwn{2}(\hsf_2) = \hsf_2$ a.e. $[\mu_2]$ must hold. Using $\hsf_2(x) = \card{\phi^{-2}(\{x\})}$ we get
\begin{align*}
    \hsf_2(3) = 1 \quad \text{ and }\quad \hsf_2(4)= 2.
\end{align*}
Hence, $\hsf_2$ is not constant on $\phi^{-2}(0)$, which means that $\efwn{2}(\hsf_2) = \hsf_2$ a.e. $[\mu_2]$ fails to hold. This shows that $\cfw$ is not spectrally 2-centered.
\end{exa}

\section{Final remarks}
This paper focused on the core structural properties and characterisations of spectrally $n$-centered weighted composition operators. Building on these results, in a forthcoming sequel, we will provide a characterisation of spectrally centered weighted composition operators. Moreover, we will apply our findings to the domains of Aluthge transforms and reciprocals (in particular, to Moore-Penrose inverses) of unbounded weighted composition operators, known to be instrumental for research in the bounded case.

\section*{Acknowledgments}
I wish to express my deep gratitude to Professor Marek Ptak and his wife for their hospitality and support during and beyond preparation of this paper. 

\section*{Statements and declarations}
This work was supported by the Ministry of Science and Higher Education of the Republic of Poland. The author declares no competing interests. No datasets were generated or analysed during the current research.

\bibliographystyle{amsalpha}

\end{document}